\documentclass[letterpaper,12pt]{amsart}
\usepackage{latexsym,array,delarray,amsmath,amsthm,amssymb,epsfig,setspace,tikz,ulem,enumitem}
\usepackage{cite}
\usepackage{float}
\usepackage{nicematrix}	
\usepackage{import}

\usepackage{color}

\theoremstyle{plain}
\newtheorem{thm}{Theorem}[section]
\newtheorem{theorem}[thm]{Theorem}
\newtheorem{lemma}[thm]{Lemma}
\newtheorem{prop}[thm]{Proposition}
\newtheorem{proposition}[thm]{Proposition}
\newtheorem{cor}[thm]{Corollary}
\newtheorem{corollary}[thm]{Corollary}

\newtheorem*{thm*}{Theorem}
\newtheorem*{lemma*}{Lemma}
\newtheorem*{prop*}{Proposition}
\newtheorem*{cor*}{Corollary}
\newtheorem*{conj*}{Conjecture}

\theoremstyle{definition}
\newtheorem{defn}[thm]{Definition}
\newtheorem{definition}[thm]{Definition}
\newtheorem{ex}[thm]{Example}

\newtheorem{rmk}[thm]{Remark}

\newcommand{\cf}{\mathcal{F}}

\newcommand{\cm}{\mathcal{M}}

\newcommand{\ind}{\mbox{$\perp \kern-5.5pt \perp$}}

\newcommand{\wg}{\widetilde{G}}
\newcommand{\wgso}{\widetilde{G}^*_{out}}

\newcommand{\dist}{\mathrm{dist}}

\newcommand{\rank}{{\rm rank}}

\title[Identifiability of linear compartmental models]{Identifiability of linear compartmental tree models and a general formula for input-output equations}

\author[Bortner]{Cashous Bortner}
\address{North Carolina State University}
\author[Gross]{Elizabeth Gross}
\address{University of Hawai`i at M\={a}noa}
\author[Meshkat]{Nicolette Meshkat}
\address{Santa Clara University}
\author[Shiu]{Anne Shiu}
\address{Texas A\&M University}
\author[Sullivant]{Seth Sullivant}
\address{North Carolina State University}

\date{\today}

\begin{document}
\maketitle

\begin{abstract}
A foundational question in the theory of linear compartmental models is how to assess whether a model is { structurally} identifiable -- that is, whether  parameter values can be inferred from noiseless data 
-- directly from the combinatorics of the model.  
{ Our main result} completely answers this question for models (with one input and one output) in which the underlying graph is a bidirectional tree; { moreover, identifiability of such models can be verified visually}.  { Models of this structure} include two families of models { often appearing} in biological applications: catenary and mammillary models.  
{ Our analysis of such models is enabled by}
two supporting results, which are { significant} in their own right.  { One result gives} the first general formula for the coefficients of input-output equations (certain equations that can be used to 
determine identifiability) { that allows for input and output to be in distinct compartments}.  { In another supporting result}, we 
prove that identifiability is preserved when a model is enlarged { and altered} in specific ways involving adding a new compartment with a bidirected edge to an existing compartment.  
\end{abstract}

\normalem


\section{Introduction} \label{sec:intro}

Compartmental models are commonly used in fields such as pharmacokinetics, ecology, and epidemiology to understand interacting groups, or compartments~\cite{godfrey}.   In pharmacokinetics, the compartments may represent tissue or tissue groups \cite{dipiro2010concepts, hedaya2012basic, tozer1981concepts, wagner1981history}; in ecology, the compartments may represent habitat zones or role in a population (e.g., forager bee and nurse bee)  \cite{gydesen1984mathematical, khoury2011quantitative, knisley2011compartmental, mulholland1974analysis};  while in epidemiology, the compartments may represent groups of infected, susceptible, and recovered individuals \cite{blackwood2018introduction, tang2020review}.  Interactions, exchanges, or flows between compartments are represented by edges between compartments, resulting in a directed graph, with distinguished nodes representing inputs, outputs, and leaks from the system.  \emph{Linear compartmental models}, which { form} the topic of this paper, are commonly used compartmental models described by a parameterized system of \emph{linear} ordinary differential equations.   

A fundamental question regarding linear compartmental models is whether or not the parameters are {\em identifiable} from a series of observations. In this paper, we give a way to visually verify when certain 
linear 
compartmental models are identifiable. To be  precise, our main theorem (Theorem \ref{thm:class_bidi}) states: \emph{A bidirectional tree model with one input and one output is generically locally identifiable if and only if the distance between the input and output is at most one and the model has either no leaks or a single leak.} 

Bidirectional tree models, or simply tree models, are linear compartmental models where the underlying directed graph is a bidirectional tree. Tree models { often appear} in applications.  Indeed,  \cite[Example 7]{MeshkatSullivantEisenberg} discusses the importance of tree models in applications, using diffusion models along rivers and streams \cite{gydesen1984mathematical}  and models of neuronal dendritic trees \cite{bressloff1993compartmental} as motivating applications.  As another example,  \cite[Example 6]{MeshkatSullivantEisenberg} considers a 11-compartment tree model, obtained by modifying a compartmental model of manganese pharmacokinetics in rats \cite{douglas2010chronic}.   

Two families of tree models { that often arise} in applications 
are catenary and mammillary models. 
For catenary (respectively, mammillary) models, 
the underlying directed graph is a path (respectively, a star). 
As corollaries to the main theorem, we give a full classification of when catenary and mammillary models are {\em generically locally identifiable} in the case of a single input and output (Corollaries~\ref{cor:cat} and~\ref{cor:mam}) . 

Generic local identifiability is a form of structural identifiability, a model property that guarantees unique parameter inference given noiseless and continuous data~\cite{Bellman}. While structural identifiability is based on perfect, i.e., noiseless data, the property is necessary for parameter estimation in the noisy setting, and thus is usually established before applying inference techniques with observed data.  

Combinatorial conditions for identifiability that can be visually verified, as in the main theorem, are desired because compartmental models are described using a graphical structure and are often used in settings with few compartments.  
Prior results in this direction were given by Cobelli {\em et al.}, who showed that mammillary and catenary models are 
identifiable when the models have a single input and output in the same compartment (specific to the respective models) and have at most one { leak~\cite{cobelli-constrained-1979}.}
Another known result asserts that models with 
inductively strongly connected graphs, 
a single input and output in a certain compartment, 
 and at most one leak are identifiable~\cite{linear-i-o, MeshkatSullivant, MeshkatSullivantEisenberg}. 
Other related results are due to 
Boukhobza {\em et al.}, who gave a graph-theoretic criterion for identifiability~\cite{boukhobza}, 
Chau, who explored properties of catenary and mammillary models
~\cite{CHAUcatenary,CHAUmamillary}, 
Delforge, who described necessary conditions for identifiability and posed conjectures on identifiability~\cite{DELFORGE1984, DELFORGE1985}, and
Vajda, who gave a condition for identifiability based on the submodels obtained by deleting one edge at a time~\cite{vajda1984}. 
{ Finally, other authors have investigated identifiability in dynamical network models that are more general than linear compartmental models, but where the network topology is still captured by a directed graph~\cite{hendrickx, cheng, legat}.}

Establishing structural identifiability of a model can be achieved by using differential algebra techniques to translate the problem to a linear algebra question~\cite{Ljung,MeshkatSullivantEisenberg}. In particular, the question of whether a given linear compartmental model is generically locally identifiable is equivalent to asking whether the Jacobian matrix of a certain coefficient map (arising from certain input-output equations) is generically full rank. 
{ {\em We give 
a general formula for the coefficients of these equations in terms of the combinatorics of the underlying directed graph associated to the model} (Theorem \ref{thm:coeff-i-o-general}). This is the second significant result of this work (after the main theorem mentioned earlier).} Previous formulas appear in \cite{singularlocus, MeshkatSullivant}, but only apply to models that satisfy certain conditions. For example,  the results in \cite{singularlocus} require the input and output to be in the same compartment. In comparison, the only condition of Theorem \ref{thm:coeff-i-o-general} is the existence of at least one input.

A general formula for coefficients allows us then to explore the effect of adding edges and moving inputs and outputs as we work towards an understanding of tree models.  Indeed, Theorem \ref{thm:coeff-i-o-general} implies that if the input and output are too far apart then the model is unidentifiable (Corollary \ref{cor:criterion-uniden}).  This result places immediate constraints on how inputs and outputs can be moved if identifiability is to be preserved, which we can glimpse in the main theorem, Theorem \ref{thm:class_bidi}, stated earlier.  
Our {final} set of results, which we summarize in Table \ref{tab:results-summary}, 
concerns operations involving moving inputs and outputs and adding leaf edges. {These results} establish situations where such operations preserve identifiability, {and
}
therefore contribute to a recent body of work aimed at understanding the effect on identifiability of adding, deleting, or moving an input, output, leak, or edge~\cite{CJSSS, Gerberding-Obatake-Shiu, linear-i-o}. { Our results also contribute to a more general body of work aimed at understanding which operations preserve a model's ``expected dimension''~\cite{MeshkatSullivant, baaijens-draisma, bortner-meshkat}.}

\begin{table}[ht]
    \centering
	\begin{tabular}{lcc}
    \hline
        {\bf Model} 
            & {\bf Operation} 
            & {\bf Result} \\ \hline
	Any  & Add leaf edge & Theorem~\ref{thm:add-leaf-edge} \\
	\hline
	Model with   & Add leaf edge at $i$, and move input & Theorem~\ref{thm:add-leaf}\\
		\quad $In=Out=\{i\} $~ & \quad or output to the new compartment & ~ \\ 
		\hline
    \end{tabular}
    \caption{Summary of results on operations preserving identifiability.  
    For an identifiable, strongly connected, linear compartmental model $\cm$ with one input, one output, and no leaks, if $\cm'$ is obtained from $\cm$ by the specified operation, then $\cm'$ is identifiable. 
{ For related prior results, we refer the reader to~\cite[Table~1]{CJSSS} and~\cite[Table~1]{linear-i-o}.
    }}
    \label{tab:results-summary}
\end{table}

The outline of our work is as follows. 
Section~\ref{sec:bkrd} introduces linear compartmental models and identifiability.  Our formula for the coefficients of input-output equations is proven in Section~\ref{sec:coeffs}.  Section~\ref{sec:add-edge} contains our results on operations that preserve identifiability.  In Section~\ref{sec:tree-models}, we classify identifiable tree models and then end with a discussion in Section~\ref{sec:discussion}.

\section{Background} \label{sec:bkrd}
This section introduces linear compartmental models and how to assess their identifiability using input-output equations. 
{ In particular, after defining linear compartmental models in Section \ref{sec:models} and introducing graph-theory terminology in Section \ref{sec:graph}, the remaining subsections, Sections~\ref{sec:i-o}--\ref{sec:iden}, review prior results on input-output equations and identifiability that serve as the foundation for our contributions in Sections~\ref{sec:coeffs}--\ref{sec:tree-models}.} 

We closely follow the notation { in~\cite{Gerberding-Obatake-Shiu,singularlocus}.}
Also, throughout this work, a {\em graph} is a finite, weighted (i.e., edge-labeled), directed multigraph.
Recall that a multigraph allows for {\em multi-edges}, that is, more than one edge with the same source and target.

\subsection{Linear compartmental models} \label{sec:models}
A {\em linear compartmental model} $\cm = (G,In,Out,Leak)$ consists of 
a (directed) graph $G = (V_G, E_G)$ without multi-edges and sets $In,Out,Leak \subseteq V_G$, which are 
called the {\em input}, {\em output}, and {\em leak} compartments, respectively.  
An edge $j \to i \in E_G$ is labeled by the parameter $a_{ij}$. 
We always assume that $Out$ is nonempty, because models with no outputs are not identifiable.  Finally, a model  $\cm = (G,In,Out,Leak)$ is {\em strongly connected} if $G$ is strongly connected { (that is, given any two vertices of $G$, there exist directed paths in each of the two directions between the two vertices)}.

As in prior works, a linear compartmental model is depicted by its graph $G$, plus leaks indicated by outgoing edges, input compartments labeled by ``in,'' and output compartments marked by this symbol: 
\begin{tikzpicture}
    \draw (0,0) circle (0.06);
    \draw (0.06,0.06) -- (.16,.19);
\end{tikzpicture}.
For instance, for the 3-compartment model $\cm= (G, In, Out, Leak)$ shown in 
Figure~\ref{fig:model-K3}, the graph $G$ is the complete directed graph on 3 nodes, $In=Out=\{1\}$, and $Leak=\{2\}$.

\begin{figure}[ht]
\begin{center}
\begin{tikzpicture}[scale=1.4]
 	\draw (0,0) circle (0.2); 
	\node[] at (0, 0) {1};
 	\draw (2,1) circle (0.2); 
    	\node[] at (2, 1) {2};
 	\draw (2,-1) circle (0.2); 
    	\node[] at (2, -1) {3};
	\draw[->] (.17,.22) -- (1.73,.92) ; 
	\node[] at (.8,.7) {$a_{21}$};
	\draw[<-] (.24,0.14) -- (1.8,.84); 
	\node[] at (1,.254) {$a_{12}$};
	\draw[->] (.17,-.22) -- (1.73,-.92) ; 
	\node[] at (.8,-.7) {$a_{31}$};
	\draw[<-] (.24,-0.14) -- (1.8,-.84); 
	\node[] at (1,-.254) {$a_{13}$};
	\draw[->] (1.95,.73) -- (1.95,-.73); 
	\node[] at (1.7,0) {$a_{32}$};
	\draw[<-] (2.05,.73) -- (2.05,-.73); 
	\node[] at (2.3,0) {$a_{23}$};
	\draw[->,thick] (-.7,.293) -- (-.26,.096);
	\node[] at (-.9,.38) {in};
	\draw[thick] (-.7,-.293) -- (-.26,-.096);
	\draw[thick] (-.75,-.316) circle (.05);
	\draw[<-,thick] (2.7,1.293) -- (2.26,1.096);
	\node[] at (2.6,1.05) {$a_{02}$};
\end{tikzpicture}
\end{center}
\caption{A linear compartmental model { with $In=Out=\{1\}$ and $Leak=\{2\}$}.}
\label{fig:model-K3}
\end{figure}
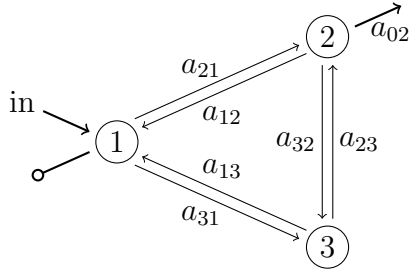

For a linear compartmental model $\mathcal{M}=(G,In,Out,Leak)$ with $n$ compartments (so, $n=|V_G|$), the \textit{compartmental matrix} $A$ is the $n \times n$ matrix defined by:
\[
A_{i,j} = \begin{cases}  
-a_{0i} - \sum_{k\colon i \to k\in E_G} a_{ki}  & i=j,\  i\in Leak, \\
-\sum_{k\colon i \to k \in E_G} a_{ki} & i=j, \ i \not\in Leak, \\
a_{ij} 		& i\neq j,~  (j,i) \in E_G, \\
0 			&  i\neq j,~  (j,i) \notin E_G. 
\end{cases}
\]
Next, the model $\mathcal M$ defines the following ODE system (\ref{eq:ode}), where $u_i(t)$ and $y_i(t)$ denote the concentrations of input and output compartments, respectively, at time $t$, and $x(t) = (x_1(t), x_2(t), \dots, x_n(t))$ is the vector of concentrations of all compartments: 
\begin{align} \label{eq:ode}
    \frac{dx}{dt} &= Ax(t) + u(t), \\
    y_i(t) &= x_i(t) \quad \quad \text{for all $i \in Out$}~, \notag
\end{align}
where $u_i(t) \equiv 0$ for $i \notin In$.

{ 
\begin{rmk} \label{rem:generic-initial-cond}
Initial conditions form an important part of an ODE system, and 
the theory of structural identifiability analysis does allow for the consideration of 
known or unknown initial conditions~\cite{saccomani2003parameter}. 
However, in this work, we assume that initial conditions are {\em generic}.
\end{rmk}
}

\subsection{Graphs associated to linear compartmental models} \label{sec:graph}
We define several auxiliary graphs arising from a linear compartmental model $\mathcal M = (G,In,Out,Leak)$.  { Examples of such graphs are shown in Figure~\ref{fig:graphs}.}
\begin{itemize}
\item 
Recall that the \textit{leak-augmented graph}~\cite{singularlocus}, denoted by $\widetilde{G}$, 
is obtained from $G$ by adding (1) a new node, labeled by $0$ and referred to as the \textit{leak node}, and 
(2) for every $j \in Leak$, an edge $j \to 0$ with label $a_{0j}$.

\item 
We introduce the graph $\widetilde{G}^*_i$ (where $i$ is some compartment), which is obtained from $\widetilde{G}$ by removing all outgoing edges from node $i$.  We also define a related matrix, denoted by $A^*_i$, 
which is obtained from the compartmental matrix $A$ of $G$ by replacing the column corresponding to compartment-$i$ with zeros.

\item 
The graph $\widetilde{G}_i$ is obtained from $\widetilde{G}^*_i$ by (1) replacing every edge 
$j \to i$ (labeled by $a_{ij}$) by the edge $j \to 0$ labeled $a_{ij}$, and then (2) deleting node $i$.
\end{itemize}

\begin{rmk} \label{rmk:multi-edges}
Among the graphs defined above, only the graph $\widetilde{G}_i$ may have multi-edges (more than one edge with the same source and target). 
Specifically, such edges may appear from a compartment to the leak node (for instance, see the graph $\widetilde{G}_1$ in Figure~\ref{fig:graphs}).
\end{rmk}

The \textit{productivity} of a graph $H$ with edge set $E_H$ is the product of its edge labels:
\begin{align} \label{eq:productivity}
	\pi_H ~:=~ \prod_{ e \in E_H } L(e)~,
\end{align}
where $L(e)$ is the label of edge $e$.
Following the usual convention, we define $\pi_H=1$ for graphs $H$ having no edges.

\begin{rmk} \label{rmk:graph-slightly-different} 
Our definition of $\widetilde{G}_i$ differs slightly from that in~\cite{singularlocus}.  
Here, we use multi-edges (e.g., $a_{02}$ and $a_{12}$ in $\widetilde{G}_1$ in Figure~\ref{fig:model-K3}), 
while 
the corresponding graph in~\cite{singularlocus} uses a single edge with the sum of the labels (e.g., $a_{02}+a_{12}$). 
Using multi-edges here is more convenient.  
Moreover, 
in the result from~\cite{singularlocus} that we use and improve (Proposition~\ref{prop:same-coeffs} below), 
it is straightforward to check that 
our definition of $\widetilde{G}_i$ yields the 
same sum of productivities.  
 Thus, both Proposition~\ref{prop:same-coeffs} and 
 the result in~\cite{singularlocus} are correct, 
 even with our updated definition of $\widetilde{G}_i$.  
\end{rmk}

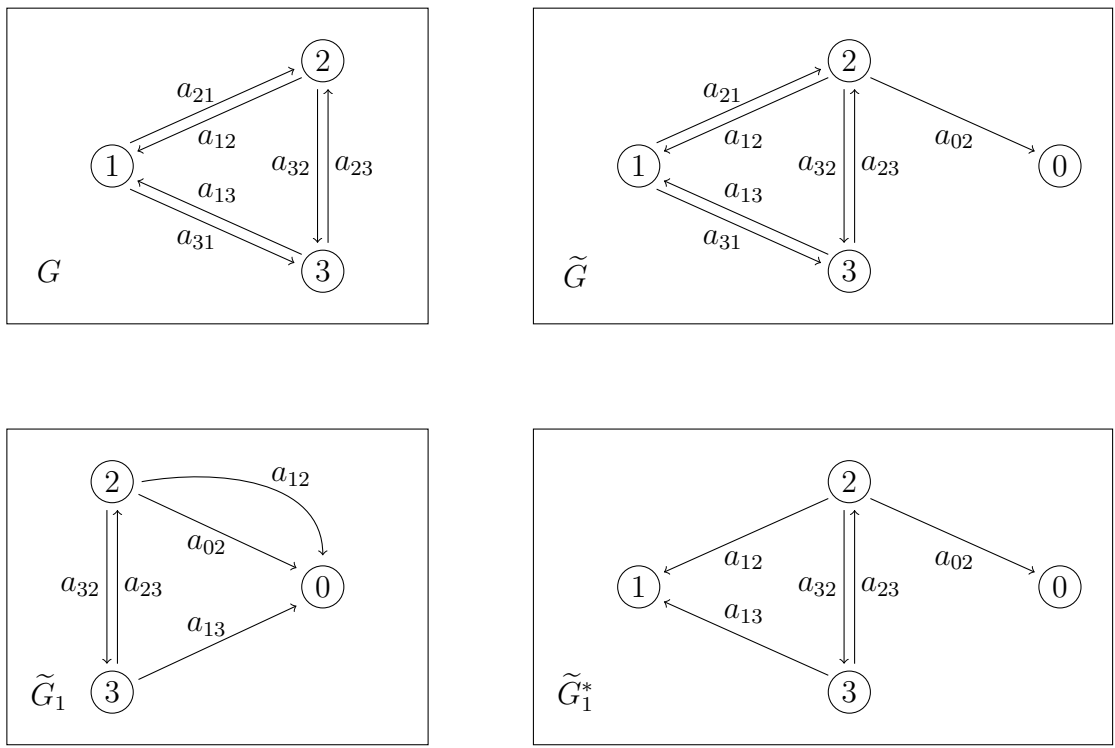
\begin{figure}[ht]
\begin{center}
\begin{tikzpicture}[scale=1.4]

 	\draw (0,0) circle (0.2); 
	\node[] at (0, 0) {1};
 	\draw (2,1) circle (0.2); 
    	\node[] at (2, 1) {2};
 	\draw (2,-1) circle (0.2); 
    	\node[] at (2, -1) {3};
	\draw[->] (.17,.22) -- (1.73,.92) ; 
	\node[] at (.8,.7) {$a_{21}$};
	\draw[<-] (.24,0.14) -- (1.8,.84); 
	\node[] at (1,.254) {$a_{12}$};
	\draw[->] (.17,-.22) -- (1.73,-.92) ; 
	\node[] at (.8,-.7) {$a_{31}$};
	\draw[<-] (.24,-0.14) -- (1.8,-.84); 
	\node[] at (1,-.254) {$a_{13}$};
	\draw[->] (1.95,.73) -- (1.95,-.73); 
	\node[] at (1.7,0) {$a_{32}$};
	\draw[<-] (2.05,.73) -- (2.05,-.73); 
	\node[] at (2.3,0) {$a_{23}$};
\draw (-1,-1.5) rectangle (3,1.5);
\node[] at (-.6,-1) {$G$};

 	\draw (0,-3) circle (0.2); 
	\node[] at (0, -3) {2};
 	\draw (0,-5) circle (0.2); 
    	\node[] at (0, -5) {3};
 	\draw (2,-4) circle (0.2); 
    	\node[] at (2, -4) {0};
	\draw[->] (.25,-3.12) -- (1.75,-3.82); 
	\node[] at (.9,-3.6) {$a_{02}$};
	\draw[->] (.25,-4.88) -- (1.75,-4.18); 
	\node[] at (.9,-4.38) {$a_{13}$}; 
	\draw[->] (-.05,-3.27) -- (-.05,-4.73); 
	\node[] at (-.3,-4) {$a_{32}$};
	\draw[<-] (.05,-3.27) -- (.05,-4.73); 
	\node[] at (.3,-4) {$a_{23}$};
	\draw[->] (.28,-3.) to[out=10, in=90] (2,-3.7); 
	\node[] at (1.7,-2.95) {$a_{12}$};
\draw (-1,-5.5) rectangle (3,-2.5);
\node[] at (-.6,-5) {$\widetilde{G}_1$};

 	\draw (5,0) circle (0.2); 
	\node[] at (5, 0) {1};
 	\draw (7,1) circle (0.2); 
    	\node[] at (7, 1) {2};
 	\draw (7,-1) circle (0.2); 
    	\node[] at (7, -1) {3};
    	\draw (9,0) circle (0.2); 
    	\node[] at (9,0) {0};
	\draw[->] (5.17,.22) -- (6.73,.92) ; 
	\node[] at (5.8,.7) {$a_{21}$};
	\draw[<-] (5.24,0.14) -- (6.8,.84); 
	\node[] at (6,.254) {$a_{12}$};
	\draw[->] (5.17,-.22) -- (6.73,-.92) ; 
	\node[] at (5.8,-.7) {$a_{31}$};
	\draw[<-] (5.24,-0.14) -- (6.8,-.84); 
	\node[] at (6,-.254) {$a_{13}$};
	\draw[->] (6.95,.73) -- (6.95,-.73); 
	\node[] at (6.7,0) {$a_{32}$};
	\draw[<-] (7.05,.73) -- (7.05,-.73); 
	\node[] at (7.3,0) {$a_{23}$};
	\draw[->] (7.2,0.84) -- (8.76,.14); 
	\node[] at (8,.254) {$a_{02}$}; 	
	
\draw (4,-1.5) rectangle (9.5,1.5);
\node[] at (4.4,-1) {$\wg$};

 	\draw (5,-4) circle (0.2); 
	\node[] at (5, -4) {1};
 	\draw (7,-3) circle (0.2); 
    	\node[] at (7, -3) {2};
 	\draw (7,-5) circle (0.2); 
    	\node[] at (7, -5) {3};
    	\draw (9,-4) circle (0.2); 
    	\node[] at (9,-4) {0};
	\draw[<-] (5.24,-3.86) -- (6.8,-3.16); 
	\node[] at (6,-3.746) {$a_{12}$};
	\draw[<-] (5.24,-4.14) -- (6.8,-4.84); 
	\node[] at (6,-4.254) {$a_{13}$};
	\draw[->] (6.95,-3.27) -- (6.95,-4.73); 
	\node[] at (6.7,-4) {$a_{32}$};
	\draw[<-] (7.05,-3.27) -- (7.05,-4.73); 
	\node[] at (7.3,-4) {$a_{23}$};
	\draw[->] (7.2,-3.16) -- (8.76,-3.86); 
	\node[] at (8,-3.746) {$a_{02}$}; 	
	
\draw (4,-5.5) rectangle (9.5,-2.5);
\node[] at (4.4,-5) {$\wg^*_1$};
\end{tikzpicture}	
\end{center}
\caption{Graphs arising from the linear compartmental model in Figure~\ref{fig:model-K3}.
}
\label{fig:graphs}
\end{figure}

\begin{ex}\label{ex:graphs} 
For the model in Figure~\ref{fig:model-K3}, the corresponding graphs $G, \wg,\widetilde{G}_1,$ and $\widetilde{G}_1^*$ are shown in Figure~\ref{fig:graphs}.  
{ The matrices arising from $G$ and  $\widetilde{G}_1^*$ are, respectively, as follows:}
{\footnotesize
\[
A=\begin{bmatrix}
-(a_{21}+a_{31}) & a_{12} & a_{13} \\
a_{21} & -(a_{02}+a_{12}+a_{32}) & a_{23} \\
a_{31} & a_{32} & -(a_{13}+a_{23})
\end{bmatrix},\
A^*_1=\begin{bmatrix}
0 & a_{12} & a_{13} \\
0 & -(a_{02}+a_{12}+a_{32}) & a_{23} \\
0 & a_{32} & -(a_{13}+a_{23})
\end{bmatrix}
\]
}
The ODE system~\eqref{eq:ode} for this model is as follows:
{\footnotesize
\begin{align*}
\begin{pmatrix}
\dot{x}_1  \\
\dot{x}_2  \\
\dot{x}_3  \\
\end{pmatrix} &= 
A
\begin{pmatrix}
x_1  \\
x_2  \\
x_3  \\
\end{pmatrix}
+
\begin{pmatrix}
u_1  \\
0  \\
0  \\
\end{pmatrix}
~=~
\begin{pmatrix}
-(a_{21}+a_{31}) x_1+ a_{12} x_2 + a_{13}x_3  + 
u_1  \\
a_{21} x_1+ -(a_{02}+a_{12}+a_{32}) x_2+ a_{23} x_3 \\
a_{31} x_1+ a_{32} x_2+ -(a_{13}+a_{23})x_3  \\
\end{pmatrix}~,
\end{align*}
}
with $y_1=x_1$.
\end{ex} 

For a graph, 
a \textit{spanning incoming forest} is a spanning subgraph for which the underlying undirected graph is a forest (i.e., has no cycles) and each node has at most one outgoing edge.  ``Spanning''
refers to the fact that every vertex of the graph is included in the forest, which
can include isolated vertices.
We introduce the following notation for a graph $H$: 
\begin{itemize}
	\item  $\mathcal{F}_{j} \left( H \right)$ is the set of all spanning incoming  forests of $H$ with exactly $j$ edges, and 
	\item  $ \mathcal{F}_j^{k, \ell } ( H )$ is the set of all spanning incoming forests of $H$ with exactly $j$ edges, such that some connected component (of the underlying undirected graph) contains both of the vertices $k$ and $\ell$. 
\end{itemize}


The following three results, which pertain to spanning incoming forests, will be used to prove the main result in Section~\ref{sec:coeffs}.

\begin{lemma} \label{lem:onesink}
Every connected component of a spanning incoming forest contains exactly one sink node, i.e., exactly one node with no outgoing edges.
\end{lemma}
\begin{proof}
Let $C$ be a connected component of a spanning incoming forest $H$ 
of a (finite) graph~$G$.
To see that a sink node exists in $C$, we start from some node in $C$ and follow outgoing arrows; eventually (as $H$ is finite and cycle-free) we must reach a sink node.  

Now assume for contradiction that $C$ has two sink nodes $v$ and $v'$.  The underlying undirected graph of $C$ is a tree, so 
it contains a unique undirected path $P$ from $v$ to $v'$.  
In the directed version of this path, each edge points in the direction of either $v$ or $v'$.  
Both $v$ and $v'$ have only incoming edges, so some node on the path $P$ has two outgoing edges -- one pointing toward $v$ and one toward $v'$.  This contradicts the fact that nodes in an incoming forest have no more than one outgoing edge.
\end{proof}


\begin{lemma} \label{lem:forest-contains-path-in-to-out}
Let $(G, In, Out, Leak)$ be a linear compartmental model.
Let $k$ and $\ell$ be distinct compartments, and let $j$ be a positive integer.
Then every forest 
$F \in \mathcal{F}_j^{k, \ell } ( \widetilde{G}_{\ell}^* )$
contains a directed path from $k$ to $\ell$.
\end{lemma}

\begin{proof}
Let $F \in  \mathcal{F}_j^{k, \ell } ( \widetilde{G}_{\ell}^* )$.  
By definition, some connected component $C$ of $F$ contains 
$k$ and $\ell$.  
By construction, 
the node $\ell$ has no outgoing edges in $\widetilde{G}_{\ell}^*$. 
So, by Lemma~\ref{lem:onesink} and its proof, 
$\ell$ is the unique sink node of $C$, and there is a directed path in $F$ from $k$ to $\ell$.
\end{proof}

The following lemma views spanning forests with a path from $k$ to $\ell$ as a union, over edges of the form $k \to i$, of forests with paths from $i$ to $\ell$. 
\begin{lemma} \label{lem:forests-as-union}
Let $H=(V_H, E_H)$ be a (directed) graph.  
Consider vertices $k, \ell \in V_H$ with $k \neq \ell$, and let $j$ be a positive integer.
Assume that $H$ has no edges outgoing from $\ell$.  
Let $K$ be the graph obtained from $H$ by removing all edges outgoing from $k$.  
Then the following equality holds:
\[
\mathcal{F}_j^{k, \ell } ( H ) ~=~ \bigcup_{i:(k \to i) \in E_H}  
							\left\{ (V_H, ~E_F \cup \{k \to i\} )
								 \mid F \in \mathcal{F}_{j-1}^{i, \ell } ( K )  \right\}~.
\]
\end{lemma}

\begin{proof}
We first prove 
``$\subseteq$''.  Let $F^* \in \mathcal{F}_j^{k, \ell } ( H ) $.  Then, $k$ and $\ell$ are in the same connected component $C$ of $F^*$. Also, by assumption, $\ell$ has no outgoing edges and so, by Lemma~\ref{lem:onesink},  
$\ell$ 
is the unique sink node of $C$.  Thus, $k$ is a non-sink node, and so there is an edge $k \to i$ in $F^*$.  Moreover, this is the unique such edge (as $F^*$ is a spanning incoming forest).

It follows that $F:= (V_H,~ E_{F^*} \smallsetminus \{k \to i\})$ is a $(j-1)$-edge, spanning subgraph of $K$.  Moreover, $F$ has no cycles and each node has at most 1 outgoing edge (because $F^*$ has the same properties).  Finally, $i$ and $\ell$ are in the same connected component of $F$ because (as we saw in the proof of Lemma~\ref{lem:onesink}) by following edges in $F^*$ we must eventually reach $\ell$, and the edge $k \to i$ is not encountered here, because otherwise $F^*$ would contain a cycle.  
We conclude that $F^* = (V_H, ~E_F \cup \{k \to i\} )$, with $F \in \mathcal{F}_{j-1}^{i, \ell } ( K ) $, as desired.

We prove 
``$\supseteq$.''  Assume that $k \to  i$ is an edge of $H$, and let $F \in  \mathcal{F}_{j-1}^{i, \ell } ( K ) $.  We must show that after adding the edge $k\to i$, the new graph $F^* :=  (V_H, ~E_F \cup \{k \to i\} )$ is in 
$
\mathcal{F}_j^{k, \ell } ( H ) $.
By construction, $F^*$ is a $j$-edge spanning subgraph of $H$.  Also, each node of $F^*$ has at most 1 outgoing edge (this property was true for $F$, and $F$ -- as a subgraph of $K$ -- had no outgoing edges from $k$).  Next, $k$ and $\ell$ are in the same connected component of $F^*$, due to the edge $k\to i$ and the fact that $i$ and $\ell$ are in the same component of $F$.  

Finally, we must show that $F^*$ has no cycles.  In $K$ (and thus also in $F$), both $k$ and $\ell$ have no outgoing edges and hence are sink nodes.  Thus, by Lemma~\ref{lem:onesink}, $k$ and $\ell$ are in distinct connected components of $F$.  Adding the edge $k\to i$ therefore joins these two components, but does not introduce any cycles.  This completes the proof.
\end{proof}


\subsection{Input-output equations} \label{sec:i-o}
In what follows, we use the following notation.  
For a matrix $B$, we let $B^{i,j}$ denote the matrix obtained from $B$ by removing row $i$ and column~$j$.  Similarly, $B^{\{i,j\},\{k,\ell\}}$ denotes the matrix obtained from $B$ by removing rows $i$ and~$j$ and columns $k$ and $\ell$.

For a linear compartmental model, an 
\emph{input-output equation} is an equation that holds along all solutions of the ODEs~\eqref{eq:ode}, and involves only the parameters~$a_{ij}$, input variables $u_i$, output variables $y_i$, and their derivatives.
One way to obtain such equations is given in the following result, which is due to Meshkat, Sullivant, and Eisenberg~\cite[Theorem 2]{MeshkatSullivantEisenberg} 
(see also~\cite[Proposition 2.3 and Remark 2.7]{linear-i-o}):
\begin{proposition}[Input-output equations, { \cite{MeshkatSullivantEisenberg} }
]  \label{prop:i-o}
Let $\mathcal M= (G, In, Out, Leak)$ be a  
 linear compartmental model with $n$ compartments and 
 at least one input.
Define $\partial I$ to be the $n \times n$ matrix in which every diagonal entry is the differential operator $d/dt$ and every off-diagonal entry is 0.  
Let $A$ be the compartmental matrix. 
Then, the following equations 
are input-output equations of $\mathcal M$:
 \begin{align} \label{eq:i-o-for-M-general}
 	\det (\partial I -A) y_i ~=~  \sum_{j \in In}  (-1)^{i+j}  \det \left( \partial I-A \right)^{j,i} u_j 
		\quad \quad {\rm for~} i \in Out~.
\end{align}
\end{proposition}

\begin{ex}[Example~\ref{ex:graphs}, continued] \label{ex:continued-i-o}
Returning to the model in Figure~\ref{fig:model-K3}, 
{ the compartmental matrix $A$ was shown in Example~\ref{ex:graphs}, which yields the following}
the input-output equation~\eqref{eq:i-o-for-M-general}: 

{\footnotesize
\begin{align*}
y_1^{(3)} &+ (a_{02}+a_{12}+a_{13}+a_{21}+a_{23}+a_{31}+a_{32})\ddot{y}_{1} +(a_{02} a_{13}+a_{12} a_{13}+a_{02} a_{21}+a_{13} a_{21}+a_{02} a_{23} +a_{12} a_{23}\\
&+a_{21} a_{23}+a_{02} a_{31}+a_{12} a_{31}+a_{23} a_{31}+a_{13} a_{32}+a_{21} a_{32}+a_{31} a_{32})\dot{y}_{1} + (a_{02} a_{13} a_{21}+a_{02} a_{21} a_{23}+a_{02} a_{23} a_{31})y_1 \\
= \ &  \ddot{u}_{1}+(a_{02}+a_{12}+a_{13}+a_{23}+a_{32})\dot{u}_1+(a_{02} a_{13}+a_{12} a_{13}+a_{02} a_{23}+a_{12} a_{23}+a_{13} a_{32})u_1.
\end{align*}
}
\end{ex}

The following result is~\cite[Theorem 4.5]{singularlocus}.  

\begin{prop}[Coefficients when input equals output, { \cite{singularlocus}}] \label{prop:same-coeffs}
Consider a 
linear compartmental model $\mathcal{M}=(G, In, Out,  Leak)$ with
$In=Out=\{1\}$. 
 Let $n$ denote the number of compartments, and 
 let $A$ be the compartmental matrix.  Write the input-output equation~\eqref{eq:i-o-for-M-general} as: 
	\begin{equation}\label{eq:inout}
	y_{1}^{(n)}+c_{n-1}y_{1}^{(n-1)}+\cdots + c_1y_{1}'+c_0y_{1}
	 ~= ~ u_{1}^{(n-1)} + d_{n-2} u_{1}^{(n-2)}+ \cdots +d_1u_{1}'+d_0u_{1}~.
	\end{equation} 
Then the coefficients of this input-output equation are as follows { (where $\pi_F$ is as in~\eqref{eq:productivity})}:
\begin{align*}
c_i &= \sum_{F \in \mathcal{F}_{n-i}(\widetilde{G})} \pi_F \quad \text{ for } i=0,1,\ldots , n-1~, \ \text{ and } \\
d_i &= \sum_{F \in \mathcal{F}_{n-i-1}(\widetilde{G}_1)} \pi_F \quad \text{ for } i=0,1,\ldots, n-2~.
\end{align*}
\end{prop}


One of the aims of this work is to generalize Proposition~\ref{prop:same-coeffs} to allow for the input and output to be in distinct compartments and for more inputs and outputs (see Theorem~\ref{thm:coeff}).

Next, we introduce the coefficient maps arising from input-output equations.  
We begin by regarding the input-output equations~\eqref{eq:i-o-for-M-general} as polynomials in the $y_j$'s and $u_i$'s and their derivatives.  Thus, each coefficient of the equation is a polynomial in the parameters ($a_{\ell m}$ for edges $m \to \ell$, and 
$a_{0 p}$ 
for leaks $p \in Leak$). 

\begin{definition} \label{def:coeff-map}
Let  $\cm=(G, In, Out, Leak)$ be a linear compartmental model.  
\begin{enumerate}[label=(\roman*)]
	\item The {\em coefficient map} $c: \mathbb{R}^{|E_G|+|Leak|} \to \mathbb{R}^m$ sends the vector of parameters to the vector of
all non-constant coefficients of
all input-output equations of the form~\eqref{eq:i-o-for-M-general}. Here, $m$ denotes the number of such coefficients. 
	\item  $\cm$ has  {\em expected dimension} if the dimension of the image of its coefficient map $c: \mathbb{R}^{|E_G|+|Leak|} \to \mathbb{R}^m$  equals 
	the minimum of $|E_G|+|Leak|$ and $m$.
\end{enumerate}
\end{definition}

\begin{rmk}
Having expected dimension is useful for proving a model has an identifiable reparametrization \cite{MeshkatSullivant}.  For example, a strongly connected model with at most $2|V_G|-2$ edges, input and output in the same compartment, and leaks from every compartment has an identifiable scaling reparametrization if and only if the model has expected dimension, which in this case is the number of independent cycles of the graph \cite[Theorem 1.2]{MeshkatSullivant}.  The case of input and output in separate compartments was analyzed in \cite{bortner-meshkat}.
\end{rmk}



\subsection{Identifiability} \label{sec:iden}
A linear compartmental model is {\em structurally identifiable} if all of its parameters can be recovered from data~\cite{Bellman}.  
Here we focus on generic local identifiability, which 
allows for recovering parameters up to a finite set, except for those in a measure-zero set of parameter space.  
This concept, in the case of strongly connected models (and others as well), is captured by the 
{ Definition~\ref{def:identifiable} (below)} 
via input-output equations (this was proven by 
Ovchinnikov, Pogudin, and Thompson~\cite[Corollary~2]{Ovchinnikov-Pogudin-Thompson}). 
{ 
This connection between identifiability and input-output equations underlies our interest in formulas for the coefficient map (as in Proposition~\ref{prop:same-coeffs}).
}

\begin{definition} \label{def:identifiable}
Consider a strongly connected linear compartmental model
 $\mathcal{M} = (G, In, Out, Leak)$ with 
at least one input.  Assume that $|E_G| + |Leak| \geq 1$.  
Let $c : \mathbb{R}^{|E_G| + |Leak|} \to \mathbb{R}^m$ be the coefficient map arising from the input-output equations~\eqref{eq:i-o-for-M-general}. Then $\mathcal{M}$ is: 
\begin{enumerate}[label=(\roman*)]
    \item \emph{generically locally identifiable} if, outside a set of measure zero, every point in $\mathbb{R}^{|E_G|+|Leak|}$ has an open neighborhood $U$
    for which the restriction $c|_U : U \to \mathbb{R}^m$ is one-to-one; and
    \item \emph{unidentifiable} if $c$ is generically infinite-to-one.
\end{enumerate}
\end{definition}
\noindent
We also adopt the convention that models $\mathcal{M} = (G, In, Out, Leak)$ without parameters, that is, 
with $|E_G| + |Leak| = 0$, are \emph{generically locally identifiable}.

\begin{ex}[Example~\ref{ex:continued-i-o}, continued] \label{ex:continued-not-iden}
For the model in Figure~\ref{fig:model-K3}, 
the input-output equation was shown in Example~\ref{ex:continued-i-o}.  
{ Following Definition~\ref{def:coeff-map},}
the resulting coefficient map $c: \mathbb{R}^7 \to \mathbb{R}^5$ is: 
\begin{multline*}
(a_{02}, a_{12}, a_{13}, a_{21}, a_{23}, a_{31}, a_{32}) \mapsto  \\
(a_{02}+ a_{12} + a_{13}+ a_{21} + a_{23} + a_{31} +  a_{32}, ~\ldots,~ a_{02} a_{13}+a_{12} a_{13}+a_{02} a_{23}+a_{12} a_{23}+a_{13} a_{32} )~.
\end{multline*}
There are more parameters than coefficients, so $c$ is generically
 infinite-to-one.  Hence, by Definition~\ref{def:identifiable}, $\cm$ is unidentifiable. 
\end{ex}

Next, we recall the following useful criteria for identifiability~\cite{MeshkatSullivantEisenberg} and expected dimension~\cite{bortner-meshkat}. 

\begin{prop}[{ \cite{bortner-meshkat, MeshkatSullivantEisenberg}}] \label{prop:MSE}
A linear compartmental model $\cm = (G, In, Out, Leak)$ 
is generically locally identifiable 
(respectively, has expected dimension) 
if and only if the rank of the Jacobian matrix of its coefficient map,  $c : \mathbb{R}^{|E_G| + |Leak|} \to \mathbb{R}^m$, when evaluated at a generic point, equals $|E_G|+|Leak|$ (respectively, equals 
the minimum of $|E_G|+|Leak|$ and $m$).
\end{prop}

Due to Proposition~\ref{prop:MSE}, we will often be interested in the ranks of Jacobian matrices, 
when evaluated at a generic point.  For brevity, we will typically omit the phrase ``when evaluated at a generic point'' and simply refer to the rank of the matrix.  
We will also use ``identifiable'' to mean ``generically locally identifiable''.

{
\begin{rmk} There are two important places where ``generic'' has a role: (1) the rank of the Jacobian matrix is evaluated at a generic point and (2) we consider models with a generic choice of initial conditions. There might be points in the parameter space where the rank of the Jacobian matrix drops and identifiability no longer holds~\cite{singularlocus}.  Likewise, there might be a choice of initial conditions where the corresponding solutions of the ODE model are not unique functions of the parameters~\cite{saccomani2003parameter}.
\end{rmk}
}

Next, we recall from~\cite{MeshkatSullivant, MeshkatSullivantEisenberg} a class of 
identifiable models $\cm = (G, In, Out, Leak)$ for which the graph $G$ is {\em inductively strongly connected}, as follows:


\begin{definition} \label{def:inductively-strongly-connected}
A graph $G$ is \textit{inductively strongly connected with respect to vertex 1} if there is a reordering of the vertices that preserves vertex 1, such that, for $i=1,2,\dots,n$, the subgraph of $G$ induced by the vertices ${\{1, 2, \dots, i\}}$ is strongly 
connected.
\end{definition}

The following result combines results from \cite{linear-i-o, MeshkatSullivantEisenberg}.

\begin{proposition}[Inductively strongly connected models] \label{prop:id-cycle-model-0-1-leaks}
Let 
$\cm=(G, In, Out, Leak)$
be
 a linear compartmental model  such that
 $In=Out=\{1\}$, $|Leak| \leq 1$, and 
 $G$ is inductively strongly connected with respect to vertex 1.  Then $\cm$ is generically locally identifiable.
\end{proposition}

\begin{proof} The model $\cm$ with $|Leak|=1$ is generically locally identifiable due to \cite[Theorem~1]{MeshkatSullivantEisenberg} and \cite[Remark 1]{MeshkatSullivantEisenberg}, and the model $\cm$ with $|Leak|= 0$ is still generically locally identifiable by \cite[Proposition 4.6]{linear-i-o} (or by definition if $G$ has no edges).
\end{proof}


Finally, we recall two additional results on adding or removing leaks~\cite[Proposition~4.6 and Theorem 4.3]{linear-i-o}, which we summarize in the following proposition. 
    \begin{proposition}[Add or remove leak, { \cite{linear-i-o}}]  \label{prop:add-remove-leak}
    Let $\mathcal M$ be a linear compartmental model that is strongly connected and has at least one input.  Assume that one of the following holds:
    \begin{enumerate}
        \item $\mathcal M$ has no leaks, and $\widetilde{\mathcal M}$ is a model obtained from $\mathcal M$ by adding one leak; or
        \item 
        $\mathcal M$ has an input, an output, and a leak in a single compartment (and no other inputs, outputs, or leaks), and $\widetilde {\mathcal M}$ is obtained from $\mathcal M$ by removing the leak.
        \end{enumerate}
    If $\mathcal M$ is generically locally identifiable, then so is $\widetilde{\mathcal M}$. 
        \end{proposition}



\section{Results on coefficients of input-output equations} \label{sec:coeffs}
The main result of this section is a combinatorial formula for the coefficients of input-output equations (Theorem~\ref{thm:coeff}).  This result generalizes Proposition~\ref{prop:same-coeffs}, which 
{applies only to 
the case with} 
input and output in the same compartment. 

\subsection{Main results} \label{sec:coeffs-results}
This subsection features our formula for the coefficients of input-output equations (Theorem~\ref{thm:coeff}), which we use to evaluate the number of non-constant coefficients of the input-output equation for strongly connected models with one input and one output (Corollary~\ref{cor:number-coefficients}).  As a consequence, we obtain a criterion for unidentifiability which arises when a model has more parameters than coefficients (Corollary~\ref{cor:criterion-uniden}).

\begin{theorem}[Coefficients of input-output equations] \label{thm:coeff}
 \label{thm:coeff-i-o-general}
Consider a linear compartmental model $\mathcal{M} = (G, In, Out, Leak)$ with at least one input.  Let $n$ denote the number of compartments.
Write the input-output equation~\eqref{eq:i-o-for-M-general} (for some $i \in Out$) as follows:
	\begin{align} \label{eq:coeff-i-o-general}
	 y_i^{(n)} + c_{n-1}  y_i^{(n-1)} + \dots  + c_1 y_i' + c_0 y_i ~=~ 
	\sum_{j \in In} (-1)^{i+j} 
		\left( 
		d_{j,n-1} u_j^{(n-1)} + 
			 \dots  + d_{j1} u_j' + d_{j0} u_j
		\right)~.
	\end{align}
	Then the coefficients of the input-output equation~\eqref{eq:coeff-i-o-general} are as follows:
\begin{align*}
	c_k ~&=~
			 \sum_{F \in \mathcal{F}_{n-k}( \widetilde G)} \pi_F   \quad \quad \quad \text{for } k=0,1,\dots,n-1~, \quad  \text{and} \\
	d_{j,k} ~&=~ 
			 \sum_{F \in \mathcal{F}^{ji}_{n-k-1}( \widetilde G^*_i)} \pi_F    \quad \quad \text{for } j\in In \text{ and } k=0,1,\dots,n-1~.
\end{align*}
\end{theorem}
The proof of Theorem~\ref{thm:coeff} is given in Section~\ref{sec:coeffs-proof}.

From Theorem~\ref{thm:coeff}, we can determine the non-constant 
coefficients in the input-output equations.  We state this result in the case of strongly connected models with one input and one output, as follows.
\begin{corollary}[Non-constant coefficients] \label{cor:which-coefficients}
Consider a strongly connected linear compartmental model $\mathcal{M} = (G, In, Out, Leak)$ with 
 $In=\{j\}$ and $Out=\{i\}$. 
Let $n$ be the number of compartments.  
Write the input-output equation~\eqref{eq:i-o-for-M-general} as follows:
	\begin{align} \label{eq:i-o-num-coeff} 
	 y_i^{(n)} + c_{n-1}  y_i^{(n-1)} + \dots  + c_1 y_i' + c_0 y_i ~=~ 
	 	 (-1)^{i+j} 
		\left( 
		d_{n-1} u_j^{(n-1)} + 
			 \dots  + d_{1} u_j' + d_{0} u_j
		\right)~.
	\end{align}
The coefficients on the left-hand side of~\eqref{eq:i-o-num-coeff} that are non-constant are as follows:
\[
	\begin{cases}
	c_0, c_1,\dots, c_{n-1} & \text{if } Leak \neq \emptyset \\
	c_1, c_2, \dots, c_{n-1}  & \text{if } Leak = \emptyset~. \\
	\end{cases}
\]
The coefficients 
	on the right-hand side of~\eqref{eq:i-o-num-coeff} that are non-constant are as follows:
\[
	\begin{cases}
	d_0, d_1,\dots, d_{n-2} & \text{if } In = Out \\
	d_0, d_1, \dots, d_{n-L-1}  & \text{if } In \neq Out~, \\
	\end{cases}
\]
where 
$L$ is the length of the shortest (directed) path from the input $j$ to the output $i$.
\end{corollary}

\begin{proof}
We first analyze the left-hand side of~\eqref{eq:i-o-num-coeff}.
By equation~\eqref{eq:i-o-for-M-general}, the coefficient $c_0$ equals, up to sign, $\det A$.  This determinant is 0 if $Leak = \emptyset$ (as $A$ in this case is the negative Laplacian of a strongly connected graph).
If, on the other hand, $Leak \neq \emptyset$, then $\det A$ is a nonzero polynomial (by~\cite[Proposition 1]{MeshkatSullivantEisenberg}) of degree $n$ in the $a_{k\ell}$'s.  

Thus, it suffices to show that $c_1, c_2, \dots, c_{n-1}$ are nonzero (they are non-constant, as their degrees are $n-1, n-2, \dots, 1$).  
As $G$ is strongly connected, there exists a spanning tree $T$ of $G$ that is directed toward compartment i (which necessarily has $(n-1)$ edges and no vertex with more than one outgoing edge).  
Let $\widetilde T$ be the corresponding subtree (with the same edges) of $\widetilde G$.  
Then, $\pi_{\widetilde T}$ is a summand of $c_1$ by Theorem~\ref{thm:coeff}.  Similarly, a summand of $c_2$ (respectively, $c_3, c_4, \dots, c_{n-1}$) is obtained by removing 1 edge (respectively, $2,3,\dots, n-2$ edges) from $\widetilde T$.  This completes the analysis of the left-hand side. 

For the right-hand side of~\eqref{eq:i-o-num-coeff}, we consider two cases. Consider first the case when $In=Out$ (i.e., $i=j$).  
By Theorem~\ref{thm:coeff},
the summands of (respectively)
$d_{n-1}, d_{n-2}, \dots, d_0$ correspond to the spanning incoming forests of $\widetilde{G}_i^*$ that have (respectively) $0, 1, \dots, n-1$ edges.  There is a unique such forest with no edges, so $d_{n-1}=1$.  Next, by construction, the tree $T$ from earlier in the proof has no edges outgoing from $i$, so we can consider the corresponding subtree (with the same edges) $\widetilde{T}^*_i$ of $\widetilde{G}_i^*$.  So, by removing (respectively) 
$0, 1, \dots, n-2$ edges from $\widetilde{T}^*_i$, we obtain a forest corresponding to a summand of (respectively) $d_0, d_1, \dots, d_{n-2}$.  Hence, $d_0, d_1,\dots, d_{n-2}$ are nonzero polynomials of degree (respectively) $n-1, n-2, \dots, 1$.  

We now consider the remaining case, when $In \neq Out$ (i.e., $i \neq j$).
First, we claim that $d_{n-1}=d_{n-2}=\dots=d_{n-L}=0$.  Indeed, by Theorem~\ref{thm:coeff} and Lemma~\ref{lem:forest-contains-path-in-to-out}, these $d_k$'s are sums over certain subgraphs of $G$, with $0,1,\dots, L-1$ (respectively) edges, containing a path from the input compartment $j$ to output $i$; but no such subgraphs exist (by definition of $L$).  
On the other hand, spanning incoming forests of $\widetilde{G}^*_i$ having 
$L, L+1, \dots, n-1$ edges 
and a directed path from the input $j$ to output $i$
do exist. 
We construct such forests as follows.
Start with a 
spanning incoming forest $F$ of $\widetilde{G}^*_i$ with $n-1$ edges (so the underlying undirected graph is a tree) such that $F$ contains a directed path $P$ of length $L$ from input to output (it is straightforward to show that such a forest exists, using the fact that $G$ is strongly connected).  Next, to obtain an appropriate forest with (respectively) 
$L, L+1, \dots, n-1$ edges, 
remove (respectively) $n-L-1, n-L-2, \dots, 0$ non-$P$ edges from $F$. 
Thus, as desired, 
 the coefficients 
 $d_{n-L-1}, d_{n-L-2}, \dots, d_0$ 
are non-constant. 
\end{proof}

\begin{rmk}[Constant coefficients] \label{rem:constant-coeff-values}
From the proof of Corollary~\ref{cor:which-coefficients}, we know the values of the constant coefficients in the input-output
equation~\eqref{eq:i-o-num-coeff}: 
\[
	\begin{cases}
	c_0=0		 & \text{if } Leak = \emptyset \\
	d_{n-1}=1 & \text{if } In = Out \\
	d_{n-L}=d_{n-L+1}= \dots = d_{n-1}=0  & \text{if } In \neq Out~. \\
	\end{cases}
\]
In particular,  in the right-hand side of~\eqref{eq:i-o-num-coeff}, 
the highest derivative $u_j^{(d)}$ 
(with nonzero coefficient)
in that sum is when $d=n-1- L$, where $L$ is the length of the shortest (directed) path from the unique input to the unique output.
\end{rmk}

Corollary~\ref{cor:which-coefficients} immediately yields the next result, which 
answers the 
question posed in~\cite[\S2.2]{Gerberding-Obatake-Shiu} 
of how read off the number of coefficients directly from a model. 
That is, we give a formula for the number $D$ where 
 $c: \mathbb{R}^{|E|+|Leak|} \to \mathbb{R}^D$ is the coefficient map.  

\begin{corollary}[Number of coefficients] \label{cor:number-coefficients}
Consider a strongly connected linear compartmental model $\mathcal{M} = (G, In, Out, Leak)$ with 
$|In|=|Out|=1$.  
Let $n$ be the number of compartments and $L$ the length of the shortest (directed) path in $G$ from the (unique) input compartment to the (unique) output. 
Then the numbers of non-constant coefficients on the left-hand and right-hand sides of~\eqref{eq:i-o-num-coeff} are as follows:
\[
\# \text{ on LHS} = \begin{cases}
n & \text{if } Leak \neq \emptyset \\
n-1 & \text{if } Leak = \emptyset 
\end{cases}
	\quad \quad
	\text{and}
	\quad \quad
\# \text{ on RHS} = \begin{cases}
n-1 		& \text{if } In=Out \\
n - L 		& \text{if } In \neq Out.
\end{cases}
\]
\end{corollary}

In the next section, we use 
Corollary~\ref{cor:number-coefficients}
to prove that identifiability is preserved when a linear compartmental model is enlarged in certain ways (see Theorems~\ref{thm:add-leaf-edge} and~\ref{thm:add-leaf}).  In~\cite{CJSSS}, 
Corollary~\ref{cor:number-coefficients} is used to partially resolve some conjectures on identifiability.

Finally, we obtain an easy-to-check condition that guarantees that a model is unidentifiable due to having more parameters than coefficients.

\begin{corollary}[Criterion for unidentifiability] \label{cor:criterion-uniden}
Consider a strongly connected linear compartmental model $\mathcal{M} = (G, In, Out, Leak)$, where $G=(V,E)$.
Assume $|In|=|Out|=1$. 
Let $n$ be the number of compartments, and let $L$ be the length of the shortest (directed) path in $G$ from the (unique) input compartment to the (unique) output. 
If one of the following conditions holds:
	\begin{enumerate}
	\item $Leak \neq \emptyset$,  $In = Out$, and $|E|+|Leak|>2n-1$,
	\item $Leak \neq \emptyset$, $In \neq Out$,  and $|E|+|Leak|>2n-L$,
	\item $Leak = \emptyset$, $In = Out$, and $|E|>2n-2$, or
	\item $Leak = \emptyset$, $In \neq Out$, and $|E|>2n-L-1$, 
	\end{enumerate}
then $\mathcal{M} $ is unidentifiable.
\end{corollary}

\begin{proof} 
First consider the case of no parameters (i.e., $|E|+|Leak| = 0$).  Then, $|E|=0 \leq 2n-2$ and (if $In \neq Out$) $|E|=0 \leq 2n-L-1$, so none of the four conditions hold.

Now assume that  $|E|+|Leak| \geq 1$. 
Let $c: \mathbb{R}^{|E|+|Leak|} \to \mathbb{R}^D$ denote the coefficient map arising from the input-output equation~\eqref{eq:i-o-for-M-general}. Corollary~\ref{cor:number-coefficients} implies that $|E|+|Leak|>D$, and so, $c$ is infinite-to-one.  Hence, $\mathcal{M}$ is unidentifiable.
\end{proof}

\begin{rmk} \label{rmk:relation-bortner-meshkat-uniden-result}
Corollary~\ref{cor:criterion-uniden} is complementary to a recent result of Bortner and Meshkat~\cite[Theorem~6.1]{bortner-meshkat}, a special case of which asserts that a strongly connected linear compartmental model with $|In|= |Out| = 1$ and $|Leak| > |In \cup Out|$, is unidentifiable.
\end{rmk}

\begin{ex}[Example~\ref{ex:continued-not-iden}, continued] \label{ex:continued-use-cor}
The model in Figure~\ref{fig:model-K3} has $n=3$ compartments, 
{ $Leak \neq \emptyset$}, 
$In=Out= \{1\}$, and $|E|+|Leak|=6+1 { = } 7 > 5 = 2n-1$.  So, 
{ Corollary~\ref{cor:criterion-uniden}} confirms what we saw in 
Example~\ref{ex:continued-not-iden}: the model is unidentifiable.
\end{ex}

\begin{ex}[Bidirectional cycle models] \label{ex:bidir-cycle} 
Let $n \geq 3$.  
Let $G_n$ be the bidirectional cycle graph with $n$ vertices (so the edges are $1 \leftrightarrows 2 \leftrightarrows \dots \leftrightarrows n \leftrightarrows 1$).  
This graph has $2n$ edges, so Corollary~\ref{cor:criterion-uniden} implies that every linear compartmental model $\cm=(G_n, In, Out, Leak)$ with $|In|=|Out|=1$ -- such as 
the model in Figure~\ref{fig:model-K3} -- is unidentifiable.  
\end{ex}

The next example shows that, in general, the converse of Corollary~\ref{cor:criterion-uniden} does not hold.

\begin{ex}
\label{ex:uniden-dist-0-leak-0} 
The model displayed below has $n=3$ compartments, $In=Out$, $Leak = \emptyset$, and $|E|=4=2n -2$.  Thus, Corollary~\ref{cor:criterion-uniden} does not apply.  Nevertheless, it is straightforward to check that the model is unidentifiable. 

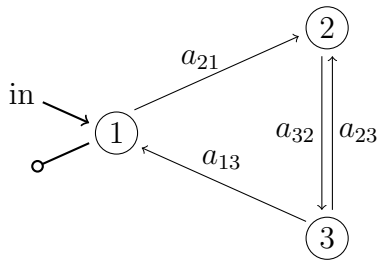
\begin{center}
\begin{tikzpicture}[scale=1.4]
 	\draw (0,0) circle (0.2); 
	\node[] at (0, 0) {1};
 	\draw (2,1) circle (0.2); 
    	\node[] at (2, 1) {2};
 	\draw (2,-1) circle (0.2); 
    	\node[] at (2, -1) {3};
	\draw[->] (.17,.22) -- (1.73,.92) ; 
	\node[] at (.8,.7) {$a_{21}$};
	\draw[<-] (.24,-0.14) -- (1.8,-.84); 
	\node[] at (1,-.254) {$a_{13}$};
	\draw[->] (1.95,.73) -- (1.95,-.73); 
	\node[] at (1.7,0) {$a_{32}$};
	\draw[<-] (2.05,.73) -- (2.05,-.73); 
	\node[] at (2.3,0) {$a_{23}$};
	\draw[->,thick] (-.7,.293) -- (-.26,.096);
	\node[] at (-.9,.38) {in};
	\draw[thick] (-.7,-.293) -- (-.26,-.096);
	\draw[thick] (-.75,-.316) circle (.05);
\end{tikzpicture}
\end{center}

\end{ex}

\subsection{Proof of Theorem~\ref{thm:coeff} } \label{sec:coeffs-proof}
To prove Theorem~\ref{thm:coeff}, we need several preliminary results.
\begin{lemma} \label{lem:remove-2-row-col}
Consider a linear compartmental model 
$\mathcal M = (G, In, Out, Leak)$ with compartmental matrix $A$. 
Let $i$ and $j$ be distinct compartments with $i \neq 1$ and $j \neq 1$.
Then:
	\begin{align*}
	\det \left( ( \lambda I - A)^{\{1, i\}, \{1 ,j\}} \right) 
		~=~
		 \lambda^{-1} \det \left( (\lambda I - A^*_{1})^{i, j} \right) .
	\end{align*}
\end{lemma}

\begin{proof}
Recall that $A^*_1$ is obtained from $A$ by replacing the first column by a column of 0's.  
Thus, the first column of $(\lambda I - A^*_1)^{i, j}$ is $(\lambda, 0, \dots, 0)^T$ (we are also using $1 \neq i,j $ here), and so  Laplace expansion along that column yields the following equality:
\begin{align} \label{eq:eqn-with-dets}
	 \det \left( (\lambda I - A^*_{1})^{i, j} \right) ~ &= ~
	 	\lambda \det \left( ( \lambda I - A^*_1)^{\{1,i\}, \{ 1,j\}} \right)\\ 
		~ &= ~
	 	\lambda \det \left( ( \lambda I - A)^{\{1,i\}, \{ 1,j\}} \right), \notag
\end{align}
and the second equality comes from the fact that, after removing column-$1$, the matrices $A$ and $A^*_1$ (and thus also $\lambda I - A$ and $\lambda I - A^*_1$) are equal. The equalities~\eqref{eq:eqn-with-dets} now imply the desired equality.
\end{proof}

\begin{lemma}\label{lem:prodeq}
Consider a linear compartmental model $\mathcal{M}=\{G,In, Out,Leak\}$ with $In=Out=\{1\}$.
Then, for every positive integer $j$, the following equality holds: 
\[
\sum_{F^* \in \mathcal{F}_j^{1,1} \left( \widetilde{G}^*_1 \right)} \pi_{F^*}~ =~ \sum_{F \in \mathcal{F}_j \left( \widetilde{G}_1 \right)} \pi_F~.
\]
\end{lemma}

\begin{proof}
First, for any graph $H$, note that $ \mathcal{F}_j^{i,i}(H)$, i.e., the $j$-edge, spanning, incoming forests of $H$ containing a path from $i$ to $i$, is the same as $\mathcal{F}_j(H)$, i.e., the $j$-edge, spanning, incoming forests of $H$.  
Hence, to complete the proof, 
it suffices to find a bijection of the following form that preserves productivity (that is, $\pi_{\phi(F^*)} = \pi_{F^*}$):
\begin{align} \label{eq:bijection-forests}
	\phi: ~ \mathcal{F}_j( \widetilde{G}^*_1  ) \to 
		\mathcal{F}_j( \widetilde{G}_1  )~. 	
\end{align}

We first explain informally what this map $\phi$ will be.  Recall that 
$\widetilde{G}_1$ is obtained from $
\widetilde{G}^*_1   $
by ``flipping'' all edges pointing toward compartment-$1$ (e.g., $2 \to 1$ and $3 \to 1$ in the lower-right of Figure~\ref{fig:graphs}) 
so that they point toward compartment-$0$ (e.g., $2\to 0$ and $3 \to 0$ in the lower-left of Figure~\ref{fig:graphs}), while keeping the same edge labels.  Accordingly, we will define $\phi$ to do the same edge-flipping in spanning forests $F^*$ of $\widetilde{G}^*_1   $ in order to obtain (as we will show) spanning forests of $\widetilde{G}_1   $.

We define $\phi$ precisely, as follows.  Let $\mathcal{L}$ denote the set of edge labels of $ \widetilde{G}_1   $ (which is also  the set of edge labels of $ \widetilde{G}^*_1   $).  A spanning subgraph (of any graph) is uniquely determined by its set of edges, so every size-$j$ subset of labels $S \subseteq \mathcal L$ defines
 (i) a unique $j$-edge subgraph of $  \widetilde{G}_1   $, which we denote by $F_S$, and also 
 (ii) a unique  $j$-edge subgraph of $  \widetilde{G}^*_1   $, which we denote by $F^*_S$.  By construction, $F_S$ and $F^*_S$ have the same productivity (for any $S \subseteq \mathcal L$).  Hence, we define $\phi$ by $\phi:F_S^* \mapsto F_S$, and then to show that this map gives the desired bijection~\eqref{eq:bijection-forests}, 
 we need only prove the following two claims:
 
 \noindent
 {\bf Claim 1}: If $F_S^* \in   \mathcal{F}_j( \widetilde{G}^*_1  ) $, then 
 each node of $F_S$ has at most 1 outgoing edge and 
 there is no cycle in the underlying undirected graph of $F_S$ .
 
 \noindent
 {\bf Claim 2}: If $F_S \in   \mathcal{F}_j( \widetilde{G}_1  ) $, then 
 each node  of $F^*_S$ has at most 1 outgoing edge and
 there is no cycle in the underlying undirected graph of $F^*_S$ .
 
The condition on the outgoing edges in Claims~1 and~2 is easy to verify.  Indeed, the edge-flip procedure preserves the source node of each edge and so the number of outgoing edges of each node is the same in $F_S$ and $F^*_S$ (or, in the case of node~$1$, there are no outgoing edges in $F_S^*$ while the node simply does not exist in $F_S$).

We prove the rest of Claims~1 and~2 by contrapositive, as follows.  Assume that $F_S$ is a subgraph of 
$\widetilde{G}_1  $ such that (i) each node has at most 1 outgoing edge and (ii) the underlying undirected graph contains a cycle.  It follows that this cycle must in fact form a directed cycle, and so must not involve node-$0$.  Hence, the edges of the cycle are not affected by edge-flipping, and so $F^*_S$ contains the same cycle.  Similarly, if $F^*_S$ is a subgraph of 
$\widetilde{G}^*_1  $ with each node having at most 1 outgoing edge and containing a cycle, then this must be a directed cycle which therefore avoids nodes $0$ and $1$, and so is present in $F_S$.

Hence, Claims~1 and~2 hold, and so we have the required bijection $\phi$ as in~\eqref{eq:bijection-forests}.
\end{proof}

%

\begin{prop} \label{prop:RHS}  
Let $\mathcal M = (G, In, Out, Leak)$ be a linear compartmental model with $n$ compartments and compartmental matrix $A$.
Let $q$ and $r$ 
be compartments.  Then, in the following equation: 
\begin{equation}\label{eq:RHS}
\det \left( (\lambda I-A)^{r, q } \right) 
= c_{n-1} \lambda^{n-1} + c_{n-2} \lambda^{n-2} + \dots + c_0~,
\end{equation} 
the coefficients are given by  
\begin{align} \label{eq:coeff-c_i}
c_k ~=~ (-1)^{q + r } 
\sum_{F \in \mathcal{F}^{r,q}_{n-k-1}(\widetilde{G}^*_{q})} 
	\pi_F 
	\quad {\rm for~} k=0,1, \dots, n-1. 
\end{align}
\end{prop}

\begin{proof}  
For convenience, we rename $out:=q$.  
Next, we claim that it suffices to consider the case of $r=1$.  Indeed, if $r \neq 1$, then switching (relabeling) compartments $1$ and $r$ (without relabeling edges) yields a 
model for which the 
compartmental matrix, which we denote by $B$, is obtained from $A$ by switching rows $1$ and $r$ and columns $1$ and $r$, and so $(\lambda I - A)^{r, out}$ and $(\lambda I - B)^{1, out}$ have the same determinant.  Thus, the $r \neq 1$ case reduces to the $r=1$ case, and so we assume $r=1$ for the rest of the proof.
%

We first analyze the case when $out=1$. 
Then, by Proposition~\ref{prop:same-coeffs}, the coefficients $c_k$ in~\eqref{eq:RHS} (for $k=0,1,\dots, n-1$) are given by the first equality here:
\[ 
c_k 
	 \quad =  \quad 
	  (-1)^{1+1} 
	 \sum_{F\in \mathcal{F}_{n-k-1} \left( \widetilde{G}_1 \right)} \pi_F
	 \quad =  \quad 
	\sum_{F \in \mathcal{F}^{1,1}_{n-k-1} \left( \widetilde{G}_1^* \right)} \pi_F~,
\]
and the second equality comes from Lemma \ref{lem:prodeq}.  This completes the case of $out=1$.

Now suppose that $out \neq 1$. We proceed by strong induction on  
the number of edges of $G$.
%
For the base case, suppose that $G$ has no edges.  Then the only edges of 
$\wgso$ (if any) are leak edges ($\ell \to 0$ for $\ell \in Leak$).  Thus, there are no spanning incoming forests on $\wgso$ in which $out$ and $1$ are in the same connected component (recall that $1 \neq out$).  
The formula in equation~\eqref{eq:coeff-c_i} therefore yields $c_0=c_1=\dots=c_{n-1}=0$.


Thus, it suffices (for the base case) to show that $\det   ( \lambda I - A )^{1,out} = 0 $.  
 To see this, note that the only nonzero entries of $A$ (if any) are leak terms on the diagonal.  Therefore 
 $(\lambda I - A)$ is also a diagonal matrix.  Hence, in the matrix $  ( \lambda I - A )^{1,out}$, the column corresponding to $1$ (which exists because $1 \neq out$) consists of 0's, and so  the determinant of $  ( \lambda I - A )^{1,out} $ is 0.  
This completes the base case.

Now suppose that the theorem holds for all models $\mathcal N = (H, In_{\mathcal N}, Out_{\mathcal N}, Leak_{\mathcal N})$ with 
$|E_H| \leq p-1$  (for some $p \geq 1$). 
Consider a model $\mathcal{M} = (G, In, Out, Leak)$ with $|E_G| = p$. 

We first consider the special case when $G$ has no edges of the form $1 \to i$, that is, outgoing from compartment-$1$.  
Essentially the same argument we made in the earlier base case applies, as follows.  
In the compartmental matrix $A$, the first column consists of 0's, 
and so 
$\det \left( ( \lambda I - A)^{1, out} \right) = 0.$  Also, there are 
 no spanning incoming forests on $\wgso$ in which $out$ and $1$ are in the same connected component (recall 
 Lemma~\ref{lem:forest-contains-path-in-to-out} and our assumption that
  $1 \neq out$).  So, equation~\eqref{eq:coeff-c_i} yields $c_0=c_1=\dots=c_{n-1}=0$.  
The theorem therefore holds in the case when $G$ has no edges outgoing from $1$.

Assume now that $G$ has at least one edge of the form $1 \to i$.  
Our first step in evaluating $\det \left( ( \lambda I - A)^{1, out} \right)$ is to perform a Laplacian expansion along the first column.  
In this column, the nonzero entries are precisely the $-a_{i,1}$'s, for those $2 \leq i \leq n$ for which $1 \to i$ is an edge (because row-1 of the matrix $(\lambda I - A)$ was deleted).  Laplace expansion along this column therefore yields the first equality here:
	\begin{align} \label{eq:pre-expand-sum}
	\notag
	 	\det\left( ( \lambda I - A)^{1, out} \right) 
		&
		~=~ 
		\sum_{i \colon (1 \to i) \in E_G } (-1)^{ i} (-a_{i 1}) \det\left( (\lambda I - A)^{\{1,i\},\{1,out\}} \right) \\
		&
		~=~ 
		 \sum_{i \colon (1 \to i) \in E_G} (-1)^{ i+1}  a_{i 1}  \lambda^{-1} \det\left( (\lambda I - A^*_{1})^{i,out} \right)~,
	\end{align}
and the second equality follows from Lemma~\ref{lem:remove-2-row-col} (and simplifying).  

Our next step is to evaluate the determinant that appears in the right-hand side of equation~\eqref{eq:pre-expand-sum}. 
Accordingly, we claim that the following equality holds:
	\begin{align} \label{eq:det-inside-sum}
	\det\left( (\lambda I - A^*_{1})^{i,out} \right)
	~=~
	(-1)^{i+out} \sum_{j=0}^{n-1} \left( \sum_{F \in \mathcal{F}_{n-j-1}^{i,out} \left( \widetilde{\mathfrak{G}}^*_{out} \right) } \pi_F \right) \lambda^j~,	
	\end{align}  
where $\mathfrak G$ is the graph obtained from $G$ by removing all edges outgoing from compartment~1.  

We will prove the claimed equality~\eqref{eq:det-inside-sum} by
interpreting the matrix $A^*_1$ as the compartmental matrix of a model having fewer edges than $\mathcal M$, and so the inductive hypothesis will apply.  To this end, notice that $A_1^*$ is the compartmental matrix of the following model:
\[
	\mathcal M_1^*  ~:=~ (\mathfrak{G}, 
			~In,~ Out, ~Leak \smallsetminus In)~.
\]

We consider two subcases, based on whether $i=out$.  The subcase when $i=out$ was proven already at the beginning of the proof (applied to the model $\mathcal M_1^*$): 

\begin{equation*}
  \det \left( (\lambda I - A^*_{1})^{out,out} \right) ~=~ \sum_{j=0}^{n-1} \left( \sum_{F \in \cf^{out, out}_{n-j-1} \left(  \widetilde{\mathfrak{G}}^*_{out} \right)} \pi_F \right) \lambda^j\ .
\end{equation*}

Now consider the remaining subcase, when $i\neq out$.
By construction
and our assumption that $G$ has an edge of the form $ 1 \to i$, 
the graph 
$\mathfrak G$ has fewer edges than $G$. 
The inductive hypothesis therefore holds for $\mathcal M_1^*$ and yields 
precisely the equality~\eqref{eq:det-inside-sum}, and so our claim is proven.

Next, we substitute the expression in~\eqref{eq:det-inside-sum} 
into the right-hand side of equation~\eqref{eq:pre-expand-sum}, simplify, rearrange the order of summation, apply 
Lemma~\ref{lem:forests-as-union} (where $H= \wgso $, $K= \widetilde{\mathfrak{G}}^*_{out} $, $k=1$, and $\ell=out$), 
and then apply the change of variables $k=j-1$:  
\begin{align*}
\det\left( ( \lambda I - A)^{1, out} \right) 
~ & = ~
 \sum_{i \colon (1 \to i)  \in E_G }  (-1)^{ i+1}  a_{i 1} \lambda^{-1}   (-1)^{i+out} \sum_{j=0}^{n-1} \left(  \sum_{F \in \mathcal{F}_{n-j-1}^{i,out} \left(  \widetilde{\mathfrak{G}}^*_{out}  \right)} \pi_F \right) \lambda^j 
\\
~ & = ~
	(-1)^{ out+1} 
	\sum_{j=0}^{n-1}
	\left(
	\sum_{i \colon (1 \to i) \in E_G }     \sum_{F \in \mathcal{F}_{n-j-1}^{i,out} \left(  \widetilde{\mathfrak{G}}^*_{out}  \right)} a_{i 1} \pi_F 
	\right)
	 \lambda^{j-1} 
\\
~ & = ~
	(-1)^{ out+1} 
	\sum_{j=0}^{n-1}
	\left(
	\sum_{F^* \in \mathcal{F}_{n-j}^{1,out} \left( \wgso \right)} \pi_{F^*}
	\right)
	 \lambda^{j-1} 
\\
~ & = ~	 
	(-1)^{ out+1} 
	\sum_{k=-1}^{n-2} 
	\left( 
	\sum_{F \in \mathcal{F}^{1,out}_{n-k-1} \left( \wgso \right)} \pi_F
	 \right) 
	\lambda^k	 ~.
\end{align*}
Comparing the above expression with the desired coefficients in~(\ref{eq:RHS}) and (\ref{eq:coeff-c_i}), it suffices to show that, 
when $k=-1$ or $k=n-1$, 
the following coefficient is 0:
\[
c_k ~=~ \sum_{F \in \mathcal{F}^{1,out}_{n-k-1} \left( \wgso \right)} \pi_F~.
\]

We first consider $k=-1$.
The graph $\wgso$
has $n+1$ nodes, and both $out$ and $0$ (the leak compartment) have no outgoing edges.  Therefore, every incoming spanning forest 
of $\wgso$
has at least two sink nodes and so (by Lemma~\ref{lem:onesink}) at least two connected components.  Such a forest therefore has no more than $n-1$ edges.  We conclude that 
 $
\mathcal{F}^{1,out}_{n-k-1} ( \wgso ) = 
{
\mathcal{F}^{1,out}_{n-(-1)-1} ( \wgso ) = 
}
\mathcal{F}^{1,out}_{n} ( \wgso ) = \emptyset
$, 
 and so $c_{-1}=0$, as desired.

Similarly, 
for $k=n-1$, we have 
$
\mathcal{F}^{1,out}_{n-k-1} ( \wgso ) = 
\mathcal{F}^{1,out}_{0} ( \wgso ) = \emptyset
$, 
because the graph with no edges lacks a path from $1$ to $out$ (recall that we have assumed $1 \neq out$).  So, $c_{n-1}=0$.
This completes the case of $1 \neq out$, and thus our proof is complete.  
\end{proof}


We can now prove Theorem~\ref{thm:coeff-i-o-general}. 

\begin{proof}[Proof of Theorem~\ref{thm:coeff-i-o-general}]
The left-hand side of the input-output equation~\eqref{eq:i-o-for-M-general} is $\det (\partial I - A) y_i$, and the formula for the coefficients 
of this expression 
was previously shown in Proposition~\ref{prop:same-coeffs}.  As for the right-hand side, the formula for these coefficients 
follows easily from Propositions~\ref{prop:i-o} and~\ref{prop:RHS}.
\end{proof}




\section{Results on adding an edge} \label{sec:add-edge}

In this section, we introduce a new operation on linear compartmental models: we add a bidirected edge from an existing compartment to a new compartment (Definition~\ref{def:add-leaf}).  
For instance, in Figure~\ref{fig:add-leaf-cycle}, 
the bidirected edge $1 \leftrightarrows 4$ is added to $\cm$ to obtain the models $\cm'$ and $\cm''$ (in $\cm'$, the output is also moved).  
We prove that identifiability is preserved when the original model has input and output in a single compartment, the new edge involves that compartment, and the input or output is moved to the new compartment (Theorem~\ref{thm:add-leaf}).
Similarly, we prove that 
identifiability is preserved when the input and output, which may be in distinct compartments, are  {\em not} moved  (Theorem~\ref{thm:add-leaf-edge}).  

\begin{rmk} \label{rem:connect-prior-work-add-edge}
{ Two related prior results also investigated the effect of adding a bidirected edge.  These results pertain to 
models that have leaks in every compartment and have expected dimension~\cite[Proposition 3.30]{bortner-meshkat} 
\cite[Proposition 5.5]{MeshkatSullivant}. }
\end{rmk}
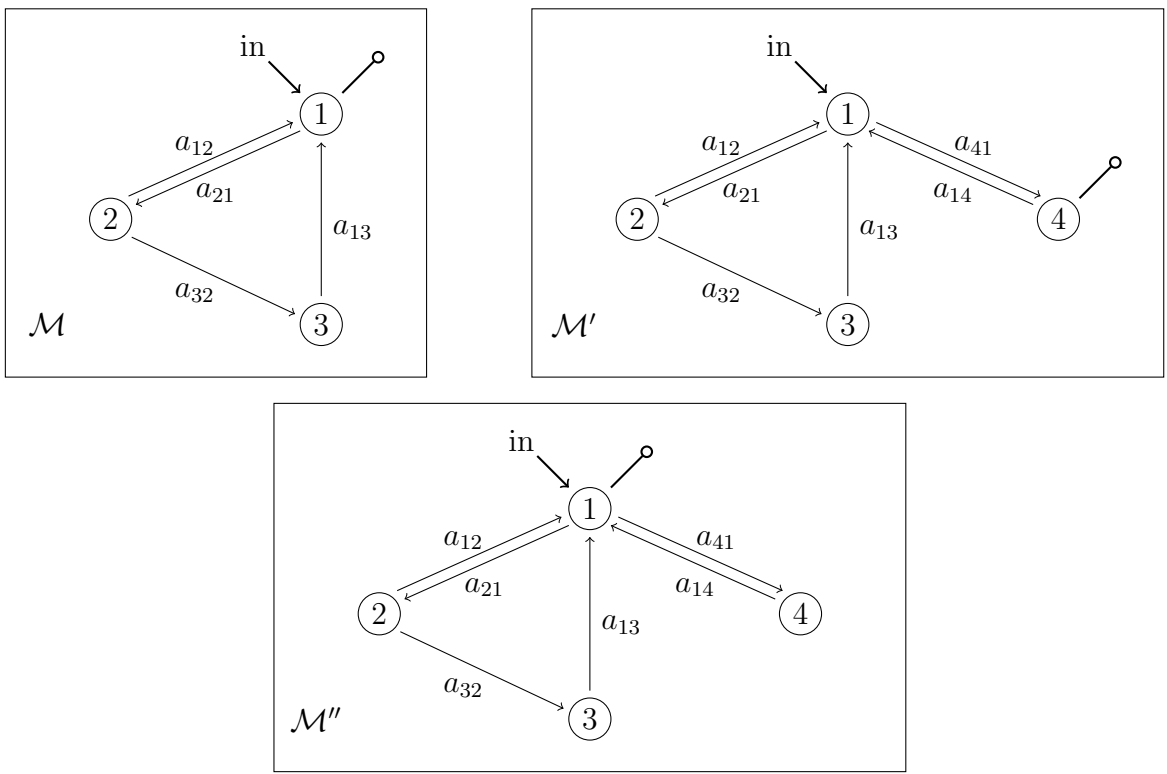
\begin{figure}[htbp]
\begin{center}
\begin{tikzpicture}[scale=1.4]
 	\draw (0,0) circle (0.2); 
	\node[] at (0, 0) {2};
 	\draw (2,1) circle (0.2); 
    	\node[] at (2, 1) {1};
 	\draw (2,-1) circle (0.2); 
    	\node[] at (2, -1) {3};
	\draw[->] (.17,.22) -- (1.73,.92) ; 
	\node[] at (.8,.7) {$a_{12}$};
	\draw[<-] (.24,0.14) -- (1.8,.84); 
	\node[] at (1,.254) {$a_{21}$};
	\draw[->] (.2,-.17) -- (1.75,-.9) ; 
	\node[] at (.8,-.7) {$a_{32}$};

	\draw[<-] (2,.73) -- (2,-.73); 
	\node[] at (2.3,-.1) {$a_{13}$};
	\draw[->,thick] (1.5,1.5) -- (1.8,1.2);
	\node[] at (1.35,1.65) {in};
	\draw[thick] (2.5,1.5) -- (2.2,1.2);
	\draw[thick] (2.54,1.54) circle (.05);

\draw (-1,-1.5) rectangle (3,2);
\node[] at (-.6,-1) {$\cm$};

 	\draw (5,0) circle (0.2); 
	\node[] at (5, 0) {2};
 	\draw (7,1) circle (0.2); 
    	\node[] at (7, 1) {1};
 	\draw (7,-1) circle (0.2); 
    	\node[] at (7, -1) {3};
    	\draw (9,0) circle (0.2); 
    	\node[] at (9,0) {4};
	\draw[->] (5.17,.22) -- (6.73,.92) ; 
	\node[] at (5.8,.7) {$a_{12}$};
	\draw[<-] (5.24,0.14) -- (6.8,.84); 
	\node[] at (6,.254) {$a_{21}$};
	\draw[->] (5.2,-.17) -- (6.75,-.9) ; 
	\node[] at (5.8,-.7) {$a_{32}$};

	\draw[<-] (7,.73) -- (7,-.73); 
	\node[] at (7.3,-.1) {$a_{13}$};
	\draw[->] (7.27,.92) -- (8.83,.22) ; 
	\node[] at (8.2,.7) {$a_{41}$};
	\draw[<-] (7.2,0.84) -- (8.76,.14); 
	\node[] at (8,.254) {$a_{14}$};
	
	\draw[->,thick] (6.5,1.5) -- (6.8,1.2);
	\node[] at (6.35,1.65) {in};
	\draw[thick] (9.5,.5) -- (9.2,.2);
	\draw[thick] (9.54,.54) circle (.05);	
	
\draw (4,-1.5) rectangle (10,2);
\node[] at (4.4,-1) {$\cm'$};
\end{tikzpicture}
\vspace*{3mm}

\begin{tikzpicture}[scale=1.4]
 	\draw (5,0) circle (0.2); 
	\node[] at (5, 0) {2};
 	\draw (7,1) circle (0.2); 
    	\node[] at (7, 1) {1};
 	\draw (7,-1) circle (0.2); 
    	\node[] at (7, -1) {3};
    	\draw (9,0) circle (0.2); 
    	\node[] at (9,0) {4};
	\draw[->] (5.17,.22) -- (6.73,.92) ; 
	\node[] at (5.8,.7) {$a_{12}$};
	\draw[<-] (5.24,0.14) -- (6.8,.84); 
	\node[] at (6,.254) {$a_{21}$};
	\draw[->] (5.2,-.17) -- (6.75,-.9) ; 
	\node[] at (5.8,-.7) {$a_{32}$};

	\draw[<-] (7,.73) -- (7,-.73); 
	\node[] at (7.3,-.1) {$a_{13}$};
	\draw[->] (7.27,.92) -- (8.83,.22) ; 
	\node[] at (8.2,.7) {$a_{41}$};
	\draw[<-] (7.2,0.84) -- (8.76,.14); 
	\node[] at (8,.254) {$a_{14}$};
	
	\draw[->,thick] (6.5,1.5) -- (6.8,1.2);
	\node[] at (6.35,1.65) {in};
	\draw[thick] (7.5,1.5) -- (7.2,1.2);
	\draw[thick] (7.54,1.54) circle (.05);
	
\draw (4,-1.5) rectangle (10,2);
\node[] at (4.4,-1) {$\cm''$};
\end{tikzpicture}
\caption{
Depicted are three models,  $\cm = (G,\{1\},\{1\},\emptyset)$, $\cm'=\{G',\{1\},\{4\},\emptyset\}$, and $\cm''=\{G',\{1\},\{1\},\emptyset\}$, where~$G'$ is the graph obtained from $G$ by adding a leaf edge at compartment $1$ (to a new compartment $4$).  
See Example~\ref{ex:add-leaf-move-output}. 
}
\label{fig:add-leaf-cycle}
\end{center}
\end{figure}

\begin{figure}[htbp]
\begin{center}
\begin{tikzpicture}[scale=1.2]
 	\draw (0,0) circle (0.2); 
	\node[] at (0, 0) {1};
 	\draw (2,0) circle (0.2); 
    	\node[] at (2, 0) {2};
 	\draw (4,0) circle (0.2); 
    	\node[] at (4, 0) {3};
	\draw[->] (.27,.05) -- (1.73,.05) ; 
	\node[] at (1,.2) {$a_{21}$};
	\draw[<-] (.27,-.05) -- (1.73,-.05) ; 
	\node[] at (1,-.2) {$a_{12}$};
	\draw[->] (2.27,.05) -- (3.73,.05) ; 
	\node[] at (3,.2) {$a_{32}$};
	\draw[<-] (2.27,-.05) -- (3.73,-.05) ; 
	\node[] at (3,-.2) {$a_{23}$};

	\draw[->,thick] (0,-.25) -> (0,-.75);
	\node[] at (.3,-.5) {$a_{01}$};

	\draw[->,thick] (-.5,.5) -- (-.2,.2);
	\node[] at (-.65,.65) {in};
	\draw[thick] (.5,.5) -- (.2,.2);
	\draw[thick] (.54,.54) circle (.05);

\draw (-1,-1) rectangle (4.5,1);
\node[] at (-.6,-.75) {$\cm$};

 	\draw (8,0) circle (0.2); 
	\node[] at (8, 0) {1};
 	\draw (10,0) circle (0.2); 
    	\node[] at (10, 0) {2};
 	\draw (12,0) circle (0.2); 
    	\node[] at (12, 0) {3};
 	\draw (6,0) circle (0.2); 
    	\node[] at (6, 0) {4};
	\draw[->] (8.27,.05) -- (9.73,.05) ; 
	\node[] at (9,.2) {$a_{21}$};
	\draw[<-] (8.27,-.05) -- (9.73,-.05) ; 
	\node[] at (9,-.2) {$a_{12}$};
	\draw[->] (10.27,.05) -- (11.73,.05) ; 
	\node[] at (11,.2) {$a_{32}$};
	\draw[<-] (10.27,-.05) -- (11.73,-.05) ; 
	\node[] at (11,-.2) {$a_{23}$};
	\draw[->] (6.27,.05) -- (7.73,.05) ; 
	\node[] at (7,.2) {$a_{14}$};
	\draw[<-] (6.27,-.05) -- (7.73,-.05) ; 
	\node[] at (7,-.2) {$a_{41}$};

	\draw[->,thick] (8,-.25) -> (8,-.75);
	\node[] at (8.3,-.5) {$a_{01}$};

	\draw[->,thick] (5.5,.5) -- (5.8,.2);
	\node[] at (5.35,.65) {in};
	\draw[thick] (8.5,.5) -- (8.2,.2);
	\draw[thick] (8.54,.54) circle (.05);

\draw (5,-1) rectangle (12.5,1);
\node[] at (5.4,-.75) {$\cm'$};

\end{tikzpicture}
\caption{
Two (catenary) models, 
$\cm = (G,\{1\},\{1\}, \{1\})$ 
and $\cm'=(G',\{4\},\{1\},\{1 \})$, 
where the graph $G'$ is obtained from $G$ by adding a leaf edge at compartment $1$.
}
\label{fig:add-leaf-cat}
\end{center}
\end{figure}
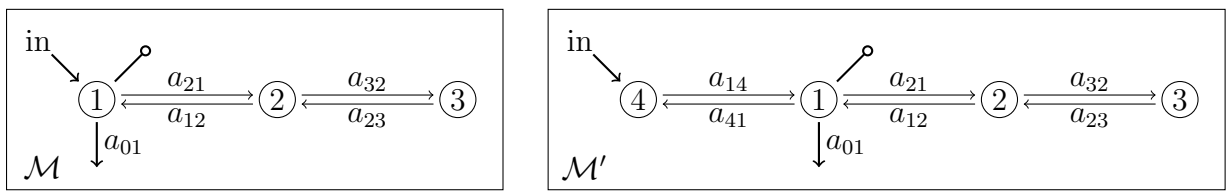

\begin{definition} \label{def:add-leaf}
Let $G=(V_G, E_G)$ be a graph with vertex set $V_G= \{1,2,\dots, n-1\}$ (for some $n \geq 2$). 
Let $i \in V_G$.  
The {\em graph obtained from $G$ by adding a leaf edge at i} is the graph $H=(V_H, E_H)$ with vertex set $V_H := \{1,2, \dots, n\}$ and edge set $E_H := E_G \cup \{ i \leftrightarrow n\}$.
\end{definition}


\begin{theorem}[Add leaf edge] \label{thm:add-leaf-edge}
Assume $n \geq 3$.  Consider a strongly connected 
linear compartmental model with $n-1$ compartments, one input, one output, and no leaks, 
$\cm=(G,\{in\},\{out\},\emptyset)$.  Let 
$H$ be the graph obtained from $G$ by adding a leaf edge at compartment $n-1$, and 
consider the linear compartmental model 
$\cm' = (H, \{in\},\{out\}, \emptyset)$. 
If $\cm$ has expected dimension (or, respectively, is generically locally identifiable), then 
$\cm'$ also has expected dimension (respectively, is generically locally identifiable).
\end{theorem}  

We prove Theorem~\ref{thm:add-leaf-edge} in Section~\ref{sec:proof-of-add-edge-result}.

\begin{theorem}[Add leaf edge and move input or output] \label{thm:add-leaf}
	Assume $n\geq 3$.
    Let $\mathcal M=(G, In, Out, Leak)$ be a strongly connected linear compartmental model 
    with $n-1$ compartments such that $In= Out = \{1\}$ and $Leak = \emptyset$.  
    Let $H$ be the graph obtained from $G$ by adding a leaf edge at compartment $1$.  
    Consider a linear compartmental model ${\mathcal M}' = (H, In', Out', Leak' )$ with 
    $Leak' = \emptyset$ and either
    $(In', Out') = (\{1\},  \{n\}) $ or $(In', Out') = ( \{n\},\{1\})$.
 Then $\mathcal M$ has expected dimension (or, respectively, is generically locally identifiable) if and only if ${\mathcal M}'$ has expected dimension (respectively, is generically locally identifiable).
\end{theorem}

We prove Theorem~\ref{thm:add-leaf} in Section~\ref{sec:proof-of-move-in-out}.  An immediate corollary, which comes from applying Proposition~\ref{prop:add-remove-leak}(1),
pertains to models with one leak, as follows.

\begin{cor} \label{cor:1-leak-add-leaf}
	Assume $n\geq 3$.
    Let $\mathcal M=(G, In, Out, Leak)$ be a strongly connected linear compartmental model with $n-1$ compartments such that $In= Out = \{1\}$ and $Leak = \emptyset$.  
    Let $H$ be the graph obtained from $G$ by adding a leaf edge at compartment $1$.  
    Consider a linear compartmental model ${\mathcal M}' = (H, In', Out', Leak' )$ with 
     $|Leak'| = 1$ and either 
    $(In', Out') = (\{1\},  \{n\}) $ or $(In', Out') = ( \{n\},\{1\})$.
        If $\cm$ is identifiable, then $\cm'$ is also identifiable.
\end{cor}

Next, we reveal 
a new class of identifiable models, namely, 
inductively strongly connected models in which the input and output compartments form a leaf edge, as follows.  

\begin{corollary}[Add a leaf and move input/output in inductively strongly connected models] \label{cor:addleaf_isc}
	Assume $n\geq 3$.
    Let $\mathcal M=(G, In, Out, Leak)$ be a linear compartmental model 
    with $n-1$ compartments such that $In= Out = \{1\}$, $Leak = \emptyset$, and $G$ is inductively strongly connected with respect to vertex $1$.  
    Let $H$ be the graph obtained from $G$ by adding a leaf edge at compartment $1$.  
    Consider a model ${\mathcal M}' = (H, In', Out', Leak' )$ with 
    $|Leak'| \leq 1$ and either
    $(In', Out') = (\{1\},  \{n\}) $ or $(In', Out') = ( \{n\},\{1\})$. 
Then $\cm '$ is generically locally identifiable.
\end{corollary}

\begin{proof} This result follows from 
Proposition \ref{prop:id-cycle-model-0-1-leaks}, Theorem \ref{thm:add-leaf}, and Corollary~\ref{cor:1-leak-add-leaf}. 
\end{proof}

\begin{rmk} \label{rem:n-at-least-3}
The assumption of $n \geq 3$ in Theorems~\ref{thm:add-leaf-edge} and~\ref{thm:add-leaf} and other results in this section is simply to avoid cases of models we are not interested in, namely, those with no compartments or no parameters.
\end{rmk} 

\begin{rmk} \label{rem:relation-to-GOS}
The effect of moving the input or output {\em without} adding new compartments or edges was
considered for cycle models in~\cite{Gerberding-Obatake-Shiu}.
\end{rmk}

\begin{rmk} \label{rem:draisma}
Baaijens and Draisma considered operations that preserve expected dimension in models with input and output in the same compartment and leaks in all compartments~\cite{baaijens-draisma}.  
\end{rmk}


\begin{ex} \label{ex:add-leaf-move-output}
Consider the models shown in Figure~\ref{fig:add-leaf-cycle}. 
The model $\mathcal M$ is identifiable by Proposition \ref{prop:id-cycle-model-0-1-leaks}.
So, by Theorems~\ref{thm:add-leaf-edge} and~\ref{thm:add-leaf}, 
 $\cm''$ and $\mathcal M'$ are also identifiable.  
 Another way to see that 
  $\mathcal M'$ is identifiable, is by applying Corollary \ref{cor:addleaf_isc} to $\mathcal{M}$. 
 
\end{ex}


\begin{ex} \label{ex:cat}
Consider the models in Figure~\ref{fig:add-leaf-cat}. 
The model $\mathcal M$ is identifiable, by Proposition \ref{prop:id-cycle-model-0-1-leaks}. 
Thus, the model obtained from $\cm$ by removing the leak, which we denote by $\cm_0$, is also identifiable, by Proposition~\ref{prop:add-remove-leak}(2).  
Applying Corollary~\ref{cor:1-leak-add-leaf} to the model $\cm_0$, we obtain that $\mathcal M'$ is also identifiable.  
\end{ex}

Theorems~\ref{thm:add-leaf-edge} and~\ref{thm:add-leaf} 
are both used in the next section to classify identifiable models in which the underlying graph is a bidirected tree.  
In particular, 
for catenary models (that is, when the graph is a path), 
we saw in Example~\ref{ex:cat} that a corollary of 
Theorem~\ref{thm:add-leaf} applies to some models with an input or output in a leaf compartment (e.g., compartments $1$ and $3$ of the model $\cm$ in Figure~\ref{fig:add-leaf-cat}),
but 
we will need Theorem~\ref{thm:add-leaf-edge} to handle models in which both the input and output are in non-leaf compartments.


The rest of this section is dedicated to proving Theorems~\ref{thm:add-leaf-edge} and~\ref{thm:add-leaf}.  
We first prove 
Theorem~\ref{thm:add-leaf-edge} (Section~\ref{sec:proof-of-add-edge-result}).  
Next, we analyze moving the output (Section~\ref{sec:move-output}) and the input (Section~\ref{sec:move-input}), and then 
combine those results to prove Theorem~\ref{thm:add-leaf} (Section~\ref{sec:proof-of-move-in-out}).

\subsection{Proof of Theorem~\ref{thm:add-leaf-edge}} \label{sec:proof-of-add-edge-result}
To prove 
Theorem~\ref{thm:add-leaf-edge}, 
we need a result from~\cite{MeshkatSullivant}. 
To state that result, 
we must first recall how a weight vector $\omega$ defines initial forms of polynomials.  
Consider a polynomial $g  \in \mathbb{K}[x_1, x_2, \dots, x_r]$, 
where $\mathbb{K}$ is a field.  Let $\omega \in \mathbb{Q}^r$.  
Then $\omega$ defines a weight of a monomial $x^{\alpha}$ (where $\alpha \in \mathbb{Z}^r_{\geq 0}$), namely, $\langle \omega, \alpha \rangle$.  
Now the {\em initial-form polynomial} (with respect to $\omega$) of $g$, denoted by $g_{\omega}$, 
is the sum of all terms of $g$ for which the monomial has highest weight.
We can now state the following lemma, which is \cite[Corollary 5.9]{MeshkatSullivant}.

\begin{lemma} \label{lemma:weightdim}
Let $\mathbb{K}$ be a field. 
Consider a map $\phi: \mathbb{K}^r \to \mathbb{K}^s$ given by polynomials $f_1, f_2, \dots, f_s \in \mathbb{K}[x_1, x_2, \dots, x_r]$.
Let $\omega \in \mathbb{Q}^r$. 
Define $\phi_{\omega}: \mathbb{K}^r \to \mathbb{K}^s$ to be the map given by the initial-form polynomials $(f_1)_{\omega}, 
(f_2)_{\omega}, \dots, (f_s)_{\omega}$.  Then 
\[
\dim(\rm{image}\ \phi_\omega) ~ \leq  ~ \dim(\rm{image}\ \phi)~.
\]
\end{lemma}


 The following proof
closely follows that of \cite[Theorem 5.7]{MeshkatSullivant}.

\begin{proof}[Proof of Theorem~\ref{thm:add-leaf-edge}]
If $in=out$, we define $D:=1$.  If $in \neq out$, we define $D$ to be the length of the shortest (directed) path in $G$ from $in$ to $out$.  By construction, if $in \neq out$, then $D$ is also the length of the shortest (directed) path from $in$ to $out$ in $H$.

Let $\phi_{\cm}$ and $\phi_{\cm'}$ denote, respectively, the coefficient maps for $\cm$ and $\cm'$.  
By Corollary~\ref{cor:number-coefficients}, the number of coefficients of $\phi_{\cm}$ is
$(n-2)+(n-1-D)=2n-3-D$.  
Similarly, the number of coefficients of $\phi_{\cm'}$ is
$2n-1-D$.  Also, 
by construction, $\cm$ has $|E_G|$ parameters; and $\cm'$ has $|E_G|+2$ parameters.  
Therefore, the assumption that $\cm$ has expected dimension 
is the following equality:
\begin{align} \label{eq:exp-dim-original}
	\dim ( \text{image } \phi_{\cm})~=~ \min \{ |E_G|,~ 2n-3-D\}~,
\end{align}
in which case our goal is to prove the following equality: 
\begin{align} \label{eq:exp-dim-new}
 \dim ( \text{image } \phi_{\cm'})~=~ \min \{ |E_G|+2,~ 2n-1-D\}~.
 \end{align}

Similarly, the assumption that $\cm$ is identifiable
is the following equality:
\begin{align} \label{eq:identifiable-original}
	\dim ( \text{image } \phi_{\cm}) =  |E_G|~,
\end{align}
in which case our goal is to prove the following equality: 
 \begin{align} \label{eq:exp-dim-new-2}
	\dim ( \text{image } \phi_{\cm'})~=~  |E_G|+2~.
	\end{align}
 The inequalities ``$\leq$'' in~\eqref{eq:exp-dim-new} and~\eqref{eq:exp-dim-new-2} always hold, so we need only prove ``$\geq$''.  Moreover, in light of the equalities~\eqref{eq:exp-dim-original} and~\eqref{eq:identifiable-original}, it suffices (for either case) to prove that 
\begin{align} \label{eq:desired-inequality-either-case}
  \dim ( \text{image } \phi_{\cm'}) 
  	~\geq ~ 
 2+ \dim ( \text{image } \phi_{\cm})~.	
\end{align}

With an eye toward applying Lemma~\ref{lemma:weightdim}, define the weight vector 
$\omega: \{a_{ij} \mid (j,i) \in E_{H}\} \to \mathbb{R}$ as follows:
\[
\omega (a_{ij}) ~:=~ 
	\begin{cases}
	0 & \text{if } (i,j) \in \{  (n-1,n) ,~ (n,n-1) \} \\ 
	1 & \text{otherwise.}
	\end{cases}
\] 
We will analyze the pullback maps 
$\phi_{\cm}^* : \mathbb{Q}[c_1,c_2,\dots, c_{n-2},~d_0,d_1,\dots, d_{n-2-D}] \to \mathbb{Q}[a_{ij} \mid (j,i) \in E_G]$ and
$\phi_{\cm'}^* : \mathbb{Q}[c_1,c_2,\dots, c_{n-1},~d_0,d_1,\dots, d_{n-1-D}] \to \mathbb{Q}[a_{ij} \mid (j,i) \in E_H ]$.  
Recall that 
$\phi_{\cm}^*$
 (respectively, $\phi_{\cm'}^*$)
 sends each $c_k$ or $d_k$ to the corresponding polynomial in the $a_{ij}$'s for the model $\cm$ (respectively, $\cm'$), as given in
Theorem~\ref{thm:coeff}.

By Theorem~\ref{thm:coeff}, all the polynomials 
 $\phi^*_{\cm } (c_i)$, 
$\phi^*_{\cm } ( d_i)$, 
$\phi^*_{\cm'}(c_i)$, 
and
$\phi^*_{\cm'}(d_i)$ 
are homogeneous in the
parameters $a_{j \ell}$.  Hence, 
the corresponding initial-form polynomials
$\phi^*_{\cm, \omega}(c_i)$, 
$\phi^*_{\cm, \omega}(d_i)$, 
$\phi^*_{\cm', \omega}(c_i)$, and 
$\phi^*_{\cm', \omega}(d_i)$ 
are obtained by removing all terms involving $a_{n-1,n}$ or $a_{n,n-1}$ -- as long as there exist other terms in the polynomial.  These other terms, by Theorem~\ref{thm:coeff}, correspond to spanning incoming forests of $H$ that do not involve the edges $(n-1) \leftrightarrows n$ (there are no leaks, so we need not leak-augment the graph), or, equivalently, spanning incoming forests of $G$.  In particular, there exist such forests of $G$ with $1,2,\dots, n-2$ edges, and so we obtain:
	\begin{align} \label{eq:c-2-to-c-n-1}
	\phi^*_{\cm', \omega} (c_i) &~=~ \phi^*_{\cm }(c_{i-1}) \quad \text{ for }  i=2,3, \ldots , n-1~. 
	\end{align}
(The shift in the index, from $i$ to $i-1$, comes from the fact that $H$ has $n$ compartments, while $G$ has $n-1$.) 
Similarly, there are
spanning incoming
forests of $G$ with $in$ and $out$ in the same component
and $D,D+1,\dots, n-2$ edges.  Thus, we have:
	\begin{align} \label{eq:d-1-and-more}
	\phi^*_{\cm' , \omega} (d_i) &~=~ \phi^*_{\cm}(d_{i-1})  \quad \text{ for }  i=1,2, \ldots , n-1-D~.
	\end{align}
There are two more coefficients of $\cm'$ to consider: $c_1$ and $d_0$.  
 By Theorem \ref{thm:coeff}, $c_1$ and $d_0$ 
 (or, more precisely, $\phi^*_{\cm' , \omega} (c_1)$
 and $\phi^*_{\cm' , \omega} (d_0)$)
 are both sums of productivities of 
 $(n-1)$-edge
 spanning incoming forests on $H$ (which has $n$ vertices).  Hence, each such forest must use exactly one edge from 
 the edges $(n-1) \leftrightarrows n$. 
 We conclude that 
 each term in 
 $\phi^*_{\cm' , \omega} (c_1)$
(respectively, in $\phi^*_{\cm' , \omega} (d_0)$) contains exactly one of $a_{n-1, n}$ or $a_{n, n-1}$.  This implies that the respective initial-form polynomials agree with the two original polynomials:
	\begin{align} \label{eq:d-0-and-c-1}
	\widetilde{c}_1 ~:=~ 
		\phi^*_{\cm', \omega} (c_1) ~=~ \phi^*_{\cm'}(c_1)  \quad \quad \text{and} \quad \quad 
	\widetilde{d}_0 ~:=~ 
		\phi^*_{\cm', \omega} (d_0) ~=~ \phi^*_{\cm'}(d_0) ~.
	\end{align}
We can say more about the polynomials $\widetilde{c}_1$ and $\widetilde{d}_0$ in~\eqref{eq:d-0-and-c-1}.  
First, $\widetilde{d}_0$ does not involve the parameter $a_{n, n-1}$, as 
$\widetilde{d}_0$ is a sum over $(n-1)$-edge 
spanning incoming forests of $H$ in which $out$ is the only sink (by Theorem~\ref{thm:coeff} and Lemma~\ref{lem:onesink}) and such forests do not contain the edge $(n-1) \to n$ (as this would make compartment-$n$ a sink).  
Moreover, it is straightforward to check that these forests are exactly those obtained by adding the edge $n \to (n-1)$ to 
an $(n-2)$-edge spanning incoming forest of $G$ in which $out$ is the only sink.  

Similarly, the $(n-1)$-edge spanning incoming forests of $H$ (with no condition on the location of the sink) that involve the edge $n \to (n-1)$ are obtained by attaching that edge to an $(n-2)$-edge spanning incoming forest of $G$.  
We summarize the above analysis as follows:
	\begin{align} \label{eq:d-0-and-c-1-again}
	\widetilde{c}_1 ~&=~ 
		a_{n-1, n} \phi^*_{\cm}(c_1) ~+~ (\text{terms involving } a_{n, n-1} \text{ but not } a_{n-1,n})~,
		 \\ 
	\notag
	\widetilde{d}_0 ~&=~ 
		a_{n-1, n} \phi^*_{\cm}(d_0)~.
	\end{align}

Let $J_{\cm}$ and $J_{\cm',\omega}$ (respectively) denote the Jacobian matrices of 
$\phi_{\cm}$ and 
$\phi_{\cm', \omega}$, where the last two rows of  $J_{\cm',\omega}$ correspond to $\widetilde{c}_1$ and $\widetilde{d}_0$, 
and the last two columns correspond to the parameters $a_{n-1, n}$ and $a_{n, n-1}$.  
We use equations~(\ref{eq:c-2-to-c-n-1}--\ref{eq:d-0-and-c-1-again}) to relate the two Jacobian matrices as follows:

\begin{align} \label{eq:relate-jacobian-wt}
J_{\cm',\omega} ~=~
\left(
	\begin{array}{ccc|cc}
	&&&0 & 0 \\	
	& J_{\cm} &  & \vdots & \vdots  \\
	&&& 0 & 0 \\	
	 \hline
	* & \dots &  * &  \frac{\partial \widetilde{c}_1}{\partial a_{n-1, n}} &   \frac{\partial \widetilde{c}_1}{\partial a_{n, n-1}} \\
	* & \dots &  * &	 \frac{\partial \widetilde{d}_0}{\partial a_{n-1, n}} &   \frac{\partial \widetilde{d}_0}{\partial a_{n, n-1}} 
	\end{array}
\right)
	 ~=~
\left(
	\begin{array}{ccc|cc}
	&&&0 & 0 \\	
	& J_{\cm} &  & \vdots & \vdots  \\
	&&& 0 & 0 \\	
	 \hline
	* & \dots &  * &  * &   \phi^*_{\cm}(c_1) \\
	* & \dots &  * &	\phi^*_{\cm}(d_0) &  0
	\end{array}
\right)~.
\end{align}

Both 
$\phi^*_{\cm}(c_1)$ and $\phi^*_{\cm}(d_0)$ are nonzero (by Corollary~\ref{cor:which-coefficients}), so 
equation~\eqref{eq:relate-jacobian-wt} implies that ${\rm rank}(J_{\cm',\omega}) = 2+{\rm rank}(J_{\cm }) $.  Hence, we obtain the equality below (and the inequality comes from Lemma~\ref{lemma:weightdim}):
\begin{align*} 
  \dim ( \text{image } \phi_{\cm'}) 
  	~\geq ~ 
  \dim ( \text{image } \phi_{\cm',\omega}) 
  	~ = ~ 
 2+ \dim ( \text{image } \phi_{\cm})~.	
\end{align*}
Thus, our desired inequality~\eqref{eq:desired-inequality-either-case} holds, and this completes the proof. 
%
\end{proof}

\begin{rmk}[Add leak] \label{rem:add-leaf-same-proof-using-weight}
Let $\cm$ be a strongly connected model with one input, one output, and no leaks. 
Theorem~\ref{thm:add-leaf-edge} shows that expected dimension is preserved when a leaf edge is added to $\cm$.  The same is true when, instead of a leaf edge, a leak is added to $\cm$.  This result can be proven in an analogous way to the proof of Theorem~\ref{thm:add-leaf-edge}, using a weight vector $\omega$ that is 0 on the new leak parameter, and 1 on all other parameters.  Another approach to proving this result is given in the proof of~\cite[Theorem 4.3]{linear-i-o}.
\end{rmk}

{
\begin{rmk} \label{rmk:multiple-in-or-out}
Theorems~\ref{thm:add-leaf-edge} and~\ref{thm:add-leaf} are stated for models with a single input and single output.  
Nevertheless, these results can be generalized to models with multiple inputs or outputs, if the corresponding models with a single input and single output are identifiable.  This is because adding inputs or outputs preserves identifiability~\cite[Proposition 4.1]{linear-i-o}.  
\end{rmk}
}
\subsection{Moving the output} \label{sec:move-output}

In this subsection, we examine what happens to a model when a leaf edge is added and the output is moved to the new compartment
(see Proposition~\ref{prop:coeffs-jac_addout_noleak}).  The key lemma we need is as follows. 

\begin{lemma} \label{lem:add-leaf-matrices} 
 Assume $n \geq 3$. 
Let $\mathcal M=(G, In, Out, Leak)$ be a linear compartmental model 
    with $n-1$ compartments such that $In= Out = \{1\}$ and $Leak = \emptyset$.  
    Let $H$ be the graph obtained from $G$ by adding a leaf edge at compartment~$1$, 
    and let ${\mathcal M}' = (H, In', Out', Leak')$ 
    be a linear compartmental model 
	with $Leak' = \emptyset$.
Let $A$ and $A^*$ (respectively) denote the compartmental matrices of $\mathcal M$ and $\mathcal{M}'$.  
Then: 
\begin{enumerate}
	\item 
	$ \det(\lambda I - A^*) ~=~
	\lambda \det( \lambda I -A )  
	+ a_{1n} \det( \lambda I - A)
	+ a_{n1} \lambda \det \left( ( \lambda I -A)^{1,1} \right)$,
	\item $ \det \left( ( \lambda I -A^*)^{1,n} \right) ~=~ (-1)^{n-1} a_{n1} \det \left( ( \lambda I -A)^{1,1} \right)$, and
	\item $ \det \left( ( \lambda I -A^*)^{n, 1} \right) ~=~ (-1)^{n-1} a_{1n} \det \left( ( \lambda I -A)^{1,1} \right)$.
\end{enumerate}	
\end{lemma}
\begin{proof}
Letting $B$ denote the matrix obtained by removing the first row from $\lambda I - A$, we have the following:
\begin{align} \notag 
\lambda I  - A ~&=~ 
	\begin{pmatrix}
	 \lambda + \sum\limits_{ (1 \to j) \in E_G} a_{j1} & -a_{12} & -a_{13} & \cdots & -a_{1(n-1)} \\
	 \hline
	\\	
	& B & & & \\
	\\
	\end{pmatrix} ~, ~
	{\rm and} \\
	\label{eq:lambda-I-minus-A*}
\lambda I - A^* ~&=~ 
\left(
	\begin{array}{ccccc|c}
	 \lambda + a_{n1} + \sum\limits_{ (1 \to j) \in E_G} a_{j1} & -a_{12} & -a_{13} & \cdots & -a_{1(n-1)} & - a_{1n} \\
	 \hline
	&&&&& 0 \\	
	& B & & & & \vdots  \\
	&&&&& 0 \\	
	\hline
	-a_{n1} & 0 & 0 & \cdots & 0 &  \lambda + a_{1n}
	\end{array}
\right)~,
\end{align}
where, for non-edges $k \to 1$, we define $a_{1k}:=0$.  Next, letting $B^{\emptyset, 1}$ denote the matrix obtained by removing the first column of $B$, we have  $B^{\emptyset, 1} = (\lambda I - A)^{1,1}$.  We will use this equality several times in the rest of the proof.

Applying a Laplace expansion along the last row of the matrix $(\lambda I - A^*)^{1,n}$ (see~\eqref{eq:lambda-I-minus-A*}), we obtain Lemma~\ref{lem:add-leaf-matrices}(2): 
\[
	\det\left( ( \lambda I - A^*)^{1,n} \right) ~=~ (-1)^{n-2} (-a_{n1}) \det (B^{\emptyset, 1}) 
	~=~ (-1)^{n-1} a_{n1} \det\left( ( \lambda I - A)^{1,1} \right)~.
\]

Similarly, a Laplacian expansion along the last column yields Lemma~\ref{lem:add-leaf-matrices}(3): 
\[
	\det\left( ( \lambda I - A^*)^{n,1} \right) 
	~=~ 
	(-1)^{n-2} (-a_{1n}) \det (B^{\emptyset, 1})
	~=~
	 (-1)^{n-1} a_{1n}  \det \left( ( \lambda I - A)^{1,1} \right)~.
\]

Finally, we prove Lemma~\ref{lem:add-leaf-matrices}(1) by expanding along the last column in~\eqref{eq:lambda-I-minus-A*} and using the linearity of the determinant:
	\begin{align*} 
	\det( \lambda I - A^*) 
	~&=~ 
	(-1)^{n-1} (-a_{1n}) (-1)^{n-2} (-a_{n1}) \det (B^{\emptyset, 1}) 
	\\
	\notag
	& \quad \quad \quad +
	(\lambda + a_{1n}) \left( \det (\lambda I - A) + \det 
		\left(
		\begin{array}{cccc}
		a_{n1} & 0 & \cdots & 0 \\
		\hline
		\\
		& B &&\\
		\\
		\end{array}
		\right)
		\right) 
	\\
	~&=~
	 -a_{1n}  a_{n1}  \det (B^{\emptyset, 1}) 
	 	+ 
	(\lambda + a_{1n}) ( \det (\lambda I - A) + a_{n1} \det (B^{\emptyset, 1})  )
	\\
	~&=~
	\lambda \det( \lambda I -A )  
	+ a_{1n} \det( \lambda I - A)
	+ a_{n1} \lambda \det \left( ( \lambda I -A)^{1,1} \right)~.
	\end{align*}
\end{proof}


\begin{prop}[Move output] \label{prop:coeffs-jac_addout_noleak}
	Assume $n\geq 3$.
    Let $\mathcal M=(G, In, Out, Leak)$ be a strongly connected linear compartmental model 
    with $n-1$ compartments such that $In= Out = \{1\}$ and $Leak = \emptyset$.  
    Let $H$ be the graph obtained from $G$ by adding a leaf edge at compartment~$1$, 
    and let ${\mathcal M}' = (H, In', Out', Leak')$ 
    be the linear compartmental model with $In' = \{1\}$, $Out' = \{n\}$, 
	and $Leak' = \emptyset$.
 Write the input-output equation~\eqref{eq:i-o-for-M-general} for $\mathcal M$ as: 
	\begin{equation*}
	y_{1}^{(n-1)}+c_{n-2}y_{1}^{(n-2)}+\cdots + c_1y_{1}'+c_0y_{1}
	 ~= ~ u_{1}^{(n-2)} + d_{n-3} u_{1}^{(n-3)}+ \cdots +d_1u_{1}'+d_0u_{1}~,
	\end{equation*} 
and define $c_{n-1}:=1$ and $d_{n-2}:=1$.  Similarly, write the input-output equation for $\mathcal{M}^*$ as: 
	\begin{equation*}
	y_{1}^{(n)} + c^*_{n-1}y_{1}^{(n-1)}+\cdots + c^*_1y_{1}'+c^*_0y_{1}
	 ~= ~ d^*_{n-2} u_{1}^{(n-2)}+ \cdots +d^*_1u_{1}'+d^*_0u_{1}~.
	\end{equation*} 
Then:
\begin{enumerate}
	\item 
		the coefficients of $\mathcal M$ and $\mathcal M^*$ are related as follows:
	\[\begin{array}{llll}
	\rm{(i)} & d_i^* &= (-1)^{n-1}a_{n1} d_{i} & \text{ for } i \in \{0,1,\ldots ,n-2\}, \\
	\rm{(ii)} & c_i^* &=  c_{i-1}+a_{1n}c_i + a_{n1}d_{i-1}   & \text{ for } i\in\{1,2, \ldots , n-1\},  \\
	\rm{(iii)} & c_0^* &= c_0 =0~.
	\end{array}\]
	\item letting $c_{\mathcal{M}}$ and $c_{\mathcal{M}^*}$ (respectively) denote the coefficient maps of  $\mathcal M$ and $\mathcal M^*$, 
	the ranks of the resulting Jacobian matrices are related by:
		\begin{align*}
		\operatorname{rank} \left( \operatorname{Jac} ( c_{\mathcal{M}^*})  \right) ~=~  \operatorname{rank} \left( \operatorname{Jac}(c_{\mathcal{M}}) \right) + 2~. 
		\end{align*}	
\end{enumerate}
\end{prop}

\begin{proof}
The input-output equations~\eqref{eq:i-o-for-M-general} for $\cm$ and $\cm^*$ are, respectively, as follows: 
\begin{align*}
\det( \lambda I - A) y_1 &= \det\left( ( \lambda I - A)^{1,1} \right) u_1~, 
\quad {\rm and} \quad 
\det( \lambda I - A^*) y_n &= \det\left( ( \lambda I - A^*)^{1,n} \right) u_1~. 
\end{align*}
Now 
Proposition~\ref{prop:coeffs-jac_addout_noleak}(1)(i--ii) follows easily from 
Lemma~\ref{lem:add-leaf-matrices}(1--2).  
Also, Proposition~\ref{prop:coeffs-jac_addout_noleak}(1)(iii) comes from the fact that the models $\cm$ and $\cm^*$ have no leaks (cf.~\cite[Remark 2.10]{Gerberding-Obatake-Shiu}).

Now we prove part (2) of the proposition. 
%
Using part (1) of the proposition, plus $c_{n-1}:=1$ and $d_{n-2}:=1$, we obtain the following
the Jacobian matrix of the coefficient map of $\cm^*$, which we denote by $J^*$:
%
\NiceMatrixOptions{code-for-first-row = \color{blue},
code-for-first-col = \color{blue}}
\[
J^* ~=~ 
\begin{pNiceArray}{CC|CC}[first-row,first-col]
& a_{n1} & a_{1n} & \text{ Parameters } a_{kj} \text{ for all } (j,k) \in E_G \\
d_{n-2}^* & (-1)^{n-1}  & 0 &  0 \quad \cdots \quad  0 \\
c_1^* & d_0 & c_1 & \left( a_{1n} \frac{\partial c_1}{\partial a_{kj}}+ a_{n1} \frac{\partial d_0}{\partial a_{kj}} \right)_{(j,k) \in E_G}\\
\hline
c_2^* & d_1 & c_2 &  
	\left( \frac{\partial c_{1}}{\partial a_{kj}}+a_{1n} \frac{\partial c_2}{\partial a_{kj}} + a_{n1} \frac{\partial d_{1}}{\partial a_{kj}}  \right)_{(j,k) \in E_G}\\
\vdots &\vdots & \vdots & \vdots&\\
c_{n-2}^* & d_{n-3} & c_{n-2} & 
	\left(  \frac{\partial c_{n-3}}{\partial a_{kj}}+a_{1n} \frac{\partial c_{n-2}}{\partial a_{kj}} + a_{n1} \frac{\partial d_{n-3}}{\partial a_{kj}}  \right)_{(j,k) \in E_G}\\
\hline 
c_{n-1}^* & 1 & 1 &
	\left(  \frac{\partial c_{n-2}}{\partial a_{kj}}  \right)_{(j,k) \in E_G} \\  
\hline
d_0^* & (-1)^{n-1} d_0 & 0 & 
	\left( (-1)^{n-1}  a_{n1}\frac{\partial d_0}{\partial a_{kj}}  \right)_{(j,k) \in E_G}\\
\vdots & \vdots & \vdots &  \vdots \\
d_{n-3}^* & (-1)^{n-1} d_{n-3}& 0 & 
	\left( (-1)^{n-1}  a_{n1} \frac{\partial d_{n-3} }{\partial a_{kj}}   \right)_{(j,k) \in E_G}
\end{pNiceArray} ~.
\]
Next, we perform the following row operations to $J^*$, where $R_k$ denotes the row of $J^*$ corresponding to the coefficient $k$:
\begin{itemize}
\item for all $i\in \{0,2, \ldots n-2\}$, replace row $R_{d^*_i}$
	by $ (-1)^{n-1} R_{d^*_i}$, 
\item for all $i\in \{1,2, \ldots n-2\}$, replace row $R_{c^*_i}$
	by $(R_{c^*_i}-R_{d^*_{i-1}})$, 
\item iteratively from $i=n-2$ down to $i=1$, 
	replace row $ R_{c^*_i}$ by $(R_{c^*_i}-a_{1n}R_{c^*_{i+1}})$, 
\item for all $i \in \{0,1, \ldots n-3\}$. 
	replace row $R_{d^*_i}$ by $\frac{1}{a_{n1}} R_{d^*_i} $.
\end{itemize}
The resulting matrix, which has the same rank as $J^*$, has the following form:

\NiceMatrixOptions{code-for-first-row = \color{blue},
code-for-first-col = \color{blue}}
\begin{align} \label{eq:block-matrix}
 \begin{pNiceArray}{CC|CC}[first-row,first-col]
& a_{n1}& a_{1n}& \\ 
d_{n-2}^* & 1 & 0 & 0\cdots 0 \\
c_1^* & 0 & \chi &  0\cdots 0\\
\hline
c_2^* & 0 & * & 
	\left( \frac{\partial c_1}{\partial a_{kj}}  \right)_{(j,k) \in E_G} \\
\vdots &\vdots & \vdots  & \vdots \\
c_{n-1}^* &0 & * &
	\left(  \frac{\partial c_{n-2}}{\partial a_{kj}}  \right)_{(j,k) \in E_G}\\  
\hline
d_0^* & \frac{1}{a_{n1}}d_0& 0 &
	\left(  \frac{\partial d_0}{\partial a_{kj}} \right)_{(j,k) \in E_G} \\
\vdots & \vdots & \vdots & \vdots \\
d_{n-3}^* & \frac{1}{a_{n1}}d_{n-3} & 0 & 
	\left( \frac{\partial d_{n-3} }{\partial a_{kj}}  \right)_{(j,k) \in E_G}
\end{pNiceArray} =\begin{pNiceArray}{CC|CC}[first-row,first-col]
& a_{n1}& a_{1n}&  \\ 
d_{n-2}^* & 1 & 0 & 0\cdots 0 \\
c_1^* & 0 & \chi & 0 \cdots 0\\
\hline
c_2^* & 0 & * & \\
\vdots &\vdots &\vdots &  \\
c_{n-1}^* &0 & * & J \\  
d_0^* & \frac{1}{a_{n1}}d_0& 0  & \\
\vdots & \vdots & \vdots  & \\
d_{n-3}^* & \frac{1}{a_{n1}}d_{n-3} & 0 &
\end{pNiceArray} ~,
\end{align} 
where 
\begin{align*}
	\chi ~&=~
	c_1 - a_{1n}\left( c_2 -  a_{1n}\left(  \cdots - a_{1n} \left(c_{n-2}- a_{1n} \right) \right) \right) 
	~=~ (-1)^{n} (a_{1n})^{n-2} + \sum_{i=1}^{n-2} (-a_{1n})^{i-1}c_i~.
\end{align*}
By construction, each $c_i$ only involves parameters $a_{kj}$ for edges $(j,k)$ in $G$, and so:
\[
	\chi|_{a_{kj}=0 \text{ for all } (j,k) \in E_G} ~=~ (-1)^{n} (a_{1n})^{n-2}.  
\]
We conclude that $\chi$ is a nonzero polynomial.  

The fact that $\chi$ is nonzero, together with the lower block diagonal structure of the matrix on the right-hand side of~\eqref{eq:block-matrix}, imply that 
$\rank(J^*) = 2+\rank(J)$, as desired. 
\end{proof}

\subsection{Moving the input} \label{sec:move-input}

In the previous subsection, we analyzed moving the output when a leaf edge is added; now we consider moving the input. 
The following result is the analogous result to Proposition~\ref{prop:coeffs-jac_addout_noleak}, and their proofs are very similar.

\begin{prop}[Move input] \label{prop:coeffs-jac_addin_noleak}
	Assume $n\geq 3$.
    Let $\mathcal M=(G, In, Out, Leak)$ be a strongly connected linear compartmental model 
    with $n-1$ compartments such that $In= Out = \{1\}$ and $Leak = \emptyset$.  
    Let $H$ be the graph obtained from $G$ by adding a leaf edge at compartment~$1$, 
    and let ${\mathcal M}' = (H, In', Out', Leak')$ 
    be the linear compartmental model with $In' = \{1\}$, $Out' = \{n\}$, 
	and $Leak' = \emptyset$.
 Write the input-output equation~\eqref{eq:i-o-for-M-general} for $\mathcal M$ as: 
	\begin{equation*}\label{eq:in-out-original-again}
	y_{1}^{(n-1)}+c_{n-2}y_{1}^{(n-2)}+\cdots + c_1y_{1}'+c_0y_{1}
	 ~= ~ u_{1}^{(n-2)} + d_{n-3} u_{1}^{(n-3)}+ \cdots +d_1u_{1}'+d_0u_{1}~,
	\end{equation*} 
and define $c_{n-1}:=1$ and $d_{n-2}:=1$.  Similarly, write the input-output equation for $\mathcal{M}^*$ as: 
	\begin{equation*}\label{eq:in-out-move-in}
	y_{1}^{(n)} + c^*_{n-1}y_{1}^{(n-1)}+\cdots + c^*_1y_{1}'+c^*_0y_{1}
	 ~= ~ d^*_{n-2} u_{1}^{(n-2)}+ \cdots +d^*_1u_{1}'+d^*_0u_{1}~.
	\end{equation*} 
Then:
\begin{enumerate}
	\item 
		the coefficients of $\mathcal M$ and $\mathcal M^*$ are related as follows:
	\[\begin{array}{llll}
		\rm{(i)} & d_i^* &= (-1)^{n-1} a_{1n} d_{i} & \text{ for } i \in \{0,\ldots ,n-2\}\\
		\rm{(ii)} & c_i^* &= c_{i-1}+a_{1n}c_i + a_{n1}d_{i-1} & \text{ for } i\in\{1,\ldots , n-1\}  \\
		\rm{(iii)} & c_0^* &= c_0 =0~.
	\end{array}\]
	\item letting $c_{\mathcal{M}}$ and $c_{\mathcal{M}^*}$ (respectively) denote the coefficient maps of  $\mathcal M$ and $\mathcal M^*$, the ranks of the resulting Jacobian matrices are related by:
		\begin{align*}
		\operatorname{rank} \left( \operatorname{Jac} ( c_{\mathcal{M}^*})  \right) ~=~ 2+ \operatorname{rank} \left( \operatorname{Jac}(c_{\mathcal{M}}) \right) ~. 
		\end{align*}	
\end{enumerate}
\end{prop}

\begin{proof}
The input-output equations~\eqref{eq:i-o-for-M-general} for $\cm$ and $\cm^*$ are, respectively, as follows: 
\begin{align*}
\det( \lambda I - A) y_1 &= \det\left( ( \lambda I - A)^{1,1} \right) u_1~, 
\quad {\rm and} \quad 
\det( \lambda I - A^*) y_1 &= \det\left( ( \lambda I - A^*)^{n,1} \right) u_n~. 
\end{align*}
Now 
Proposition~\ref{prop:coeffs-jac_addin_noleak}(1) follows easily from 
Lemma~\ref{lem:add-leaf-matrices}(1)
and  Lemma~\ref{lem:add-leaf-matrices}(3) (and, as in the proof of Proposition~\ref{prop:coeffs-jac_addout_noleak}, the fact that the models $\cm$ and $\cm^*$ have no leaks).

We use part (1) of the proposition, plus $c_{n-1}:=1$ and $d_{n-2}:=1$, to obtain the
Jacobian matrix of the coefficient map of $\cm^*$, denoted by $J^*$:
\NiceMatrixOptions{code-for-first-row = \color{blue},
code-for-first-col = \color{blue}}
\[
J^* ~=~
\begin{pNiceArray}{CC|CC}[first-row,first-col]
& a_{1n} & a_{n1} & \text{ Parameters } a_{kj} \text{ for all } (j,k) \in E_G \\
d_{n-2}^* & 	(-1)^{n-1}  & 0 &  0 \quad \cdots \quad 0 \\
c_1^* & c_1 & d_0 & 
	\left( 	(-1)^{n-1}  a_{1n} \frac{\partial c_1}{\partial a_{kj}}+ a_{n1} \frac{\partial d_0}{\partial a_{kj}} \right)_{(j,k) \in E_G}\\
\hline
c_2^* & c_2 & d_1 & 
	\left(  \frac{\partial c_{1}}{\partial a_{kj}}+a_{1n} \frac{\partial c_2}{\partial a_{kj}} + a_{n1} \frac{\partial d_{1}}{\partial a_{kj}}  \right)_{(j,k) \in E_G}\\
\vdots &\vdots & \vdots & \vdots&\\
c_{n-2}^* & c_{n-2} & d_{n-3} & 
	\left(  \frac{\partial c_{n-3}}{\partial a_{kj}}+a_{1n} \frac{\partial c_{n-2}}{\partial a_{kj}} + a_{n1} \frac{\partial d_{n-3}}{\partial a_{kj}} \right)_{(j,k) \in E_G} \\
\hline 
c_{n-1}^* & 1 & 1 &
	\left(  \frac{\partial c_{n-2}}{\partial a_{kj}}  \right)_{(j,k) \in E_G}\\  
\hline
d_0^* & 	(-1)^{n-1} d_0 & 0 & 
	\left( 	(-1)^{n-1} a_{1n}\frac{\partial d_0}{\partial a_{kj}}  \right)_{(j,k) \in E_G}\\
\vdots & \vdots & \vdots &  \vdots \\
d_{n-3}^* & 	(-1)^{n-1} d_{n-3}& 0 &
	\left( 	(-1)^{n-1}  a_{1n} \frac{\partial d_{n-3} }{\partial a_{kj}}   \right)_{(j,k) \in E_G}
\end{pNiceArray} ~.
\]

We perform row operations on $J^*$, where $R_k$ denotes the row of $J^*$ corresponding to the coefficient $k$:
\begin{itemize}
\item for all $i\in \{0,2, \ldots n-2\}$, replace row $R_{d^*_i}$
	by $ (-1)^{n-1} R_{d^*_i}$, 
\item for all $i\in \{1,2, \ldots n-2\}$, replace row $R_{c^*_i}$
	by $(R_{c^*_i}- (a_{n1}/a_{1n}) R_{d^*_{i-1}})$, 
\item iteratively from $i=n-2$ down to $i=1$, 
	replace row $ R_{c^*_i}$ by $(R_{c^*_i}- a_{1n} R_{c^*_{i+1}})$, 
\item for all $i \in \{0,1, \ldots n-3\}$, 
	replace row $R_{d^*_i}$ by $\frac{1}{a_{n1}} R_{d^*_i} $. 
\end{itemize}

The resulting matrix, which has the same rank as $J^*$, has the following form:

\NiceMatrixOptions{code-for-first-row = \color{blue},
code-for-first-col = \color{blue}}
\begin{align} \label{eq:block-matrix-2}
 \begin{pNiceArray}{CC|CC}[first-row,first-col]
& a_{1n}& a_{n1}& \\
d_{n-2}^* & 1 & 0 & 0\cdots 0 \\
c_1^* & * & x &  0\cdots 0\\
\hline
c_2^* & * & * & 
	\left( \frac{\partial c_1}{\partial a_{kj}}  \right)_{(j,k) \in E_G} \\
\vdots &\vdots & \vdots  &\vdots \\
c_{n-1}^* &* & * & 
	\left( \frac{\partial c_{n-2}}{\partial a_{kj}}  \right)_{(j,k) \in E_G}\\  
\hline
d_0^* & \frac{1}{a_{1n}}d_0& 0 &
	\left( \frac{\partial d_0}{\partial a_{kj}} \right)_{(j,k) \in E_G} \\
\vdots & \vdots & \vdots & \vdots \\
d_{n-3}^* & \frac{1}{a_{1n}}d_{n-3} & 0 & 
	\left( \frac{\partial d_{n-3} }{\partial a_{kj}} \right)_{(j,k) \in E_G}
\end{pNiceArray} 
=
\begin{pNiceArray}{CC|CC}[first-row,first-col]
& a_{1n}& a_{n1}&  \\
d_{n-2}^* & 1 & 0 & 0\cdots 0 \\
c_1^* & * & x & 0 \cdots 0\\
\hline
c_2^* & * & * & \\
\vdots &\vdots &\vdots &  \\
c_{n-1}^* &* & * & J \\  
d_0^* & \frac{1}{a_{1n}}d_0& 0  & \\
\vdots & \vdots & \vdots  & \\
d_{n-3}^* & \frac{1}{a_{1n}}d_{n-3} & 0 &
\end{pNiceArray} ~,
\end{align}
where 
\begin{align*}
	\chi ~&=~
	d_0 - a_{1n}\left( d_2 -  a_{1n}\left(  \cdots - a_{1n}\left( d_{n-3}- a_{1n} \right) \right) \right) 
	~=~ (-1)^{n} (a_{1n})^{n-2} + \sum_{i=1}^{n-2} (-a_{1n})^{i-1}d_i~.
\end{align*}
For the same reason as in the proof of Proposition~\ref{prop:coeffs-jac_addout_noleak}, $\chi$
is a nonzero polynomial.  Thus, 
from 
the lower block diagonal structure of the matrix on the right-hand side of~\eqref{eq:block-matrix-2}, we obtain the desired equality: $\rank(J^*) = 2+\rank(J)$. 
\end{proof}

\subsection{Proof of Theorem~\ref{thm:add-leaf}} \label{sec:proof-of-move-in-out}

We now apply 
Propositions~\ref{prop:coeffs-jac_addout_noleak} and~\ref{prop:coeffs-jac_addin_noleak}
to prove our result on adding a leaf edge and moving the input or output.

\begin{proof}[Proof of Theorem \ref{thm:add-leaf}]
For models $\cm$ and $\cm^*$, let $J$ and $J^*$ denote the Jacobian matrices of the respective coefficient maps.  We first examine identifiability.  By definition, $\cm$ is identifiable if and only if 
$\rank(J) =|E_G|$ (recall that $\cm$ has no leaks).  Similarly, $\cm^*$ is identifiable if and only if 
$\rank(J^*)=|E_H|$.  Now the identifiability result follows from 
Propositions~\ref{prop:coeffs-jac_addout_noleak}--\ref{prop:coeffs-jac_addin_noleak}
and the fact that (by construction) $|E_H|=2+|E_G|$.  

As for expected dimension, we first compute the number of non-constant coefficients in the coefficient map of $\cm$ (respectively, $\cm^*$), which we denote by $N_{\cm}$ (respectively, $N_{\cm^*}$.  These numbers, by a straightforward application of Corollary~\ref{cor:number-coefficients} (in particular, we use the fact that there is an edge in $\cm^*$ from input to output, and so the length of the shortest path from input to output is 1), are as follows:
\begin{align} \label{eq:num-coefs-2-models}
	N_{\cm} ~=~ 2n-4 \quad \quad \text{and} \quad \quad 
	N_{\cm^*} ~=~ 2n-2 ~.
\end{align}
Next, by Proposition~\ref{prop:MSE}, 
$\cm$ has expected dimension if and only if 
$\rank(J) = \min \{ |E_G|, N_{\cm} \}$.
Similarly, 
$\cm^*$ has expected dimension if and only if 
$\rank(J^*) = \min \{ |E_H|, N_{\cm^*} \}$.  Now, the desired result follows from 
Propositions~\ref{prop:coeffs-jac_addout_noleak}--\ref{prop:coeffs-jac_addin_noleak}
and the equalities~\eqref{eq:num-coefs-2-models}.
\end{proof}


\section{Tree Models} \label{sec:tree-models}
In this section, we introduce bidirectional tree models, and completely characterize 
which of these models with one input and one output are identifiable (Theorem~\ref{thm:class_bidi}).  As a consequence, we determine which catenary and mammillary models with one input and one output are identifiable (Corollary~\ref{cor:cat} and~\ref{cor:mam}).  
Our results therefore extend those of~\cite{cobelli-constrained-1979}, which concerned the case when the input and output are in the same compartment.

\begin{defn} \label{def:tree-model}
A \textit{bidirectional tree graph} is a graph $G$
that is obtained from an undirected tree graph by making every edge bidirected
(that is, $(i \to j) \in E_G$ implies that $( i \leftrightarrows j) \in E_G$). 
A linear compartmental model 
$\cm=(G,In,Out,Leak)$ 
is a \textit{bidirectional tree model} (or, to be succinct, a {\em tree model})  if the graph $G$ is a bidirectional tree graph.
\end{defn}

In the following theorem, which is the main result of the section, we use the notation 
$\dist_G(i,j) $ to denote the length of shortest (directed) path in $G$ from vertex $i$ to vertex $j$.

\begin{thm}[Classification of identifiable tree models]\label{thm:class_bidi}
A tree model with exactly one input and one output $\cm = (G, \{in\}, \{out\}, Leak)$ is generically locally identifiable if and only if 
$\dist_G(in,out) \leq 1$ and $|Leak|\leq 1$.
\end{thm}

\noindent
The proof of Theorem~\ref{thm:class_bidi} appears in Section~\ref{sec:proof-tree-model}.

\begin{figure}[ht]
\begin{center}
	\begin{tikzpicture}[scale=1.8]
 	\draw (-2,0) circle (0.2);	
 	\draw (-1,0) circle (0.2);	
 	\draw (1,0) circle (0.2);	
    	\node[] at (-2, 0) {1};
    	\node[] at (-1, 0) {$2$};
    	\node[] at (0, 0) {$\dots$};
    	\node[] at (1, 0) {$n$};
	 \draw[<-] (-1.7, 0.05) -- (-1.3, 0.05);	
	 \draw[->] (-1.7, -0.05) -- (-1.3, -0.05);	
	 \draw[<-] (-0.7, 0.05) -- (-0.3, 0.05);	
	 \draw[->] (-0.7, -0.05) -- (-0.3, -0.05);	
	 \draw[<-] (0.3, 0.05) -- (0.7, 0.05);	
	 \draw[->] (0.3, -0.05) -- (0.7, -0.05);	
   	 \node[] at (-1.5, 0.2) {$a_{12}$};
   	 \node[] at (-0.5, 0.2) {$a_{23}$};
   	 \node[] at (0.5, 0.2) {$a_{n-1,n}$};
   	 \node[] at (-1.5, -0.2) {$a_{21}$};
   	 \node[] at (-0.5, -0.2) {$a_{32}$};
   	 \node[] at (0.5, -0.2) {$a_{n,n-1}$};
    	\node[above] at (0.8, -0.8) {Catenary};
\draw (-2.5,-.8) rectangle (1.5, .5);
 	\draw (3,0) circle (0.2);
 	\draw (5,-1) circle (0.2);
 	\draw (5,-0.3) circle (0.2);
 	\draw (5,1) circle (0.2);
    	\node[] at (3, 0) {1};
    	\node[] at (5, -1) {2};
    	\node[] at (5, -0.3) {3};
    	\node[] at (5, 0.3) {$\vdots$};
    	\node[] at (5, 1) {$n$};
	 \draw[->] (3.25, -.2) -- (4.65, -1);
	 \draw[<-] (3.25, -.1) -- (4.65, -0.9);
	 \draw[->] (3.25, -0.05) -- (4.65, -0.3);
	 \draw[<-] (3.25, 0.05) -- (4.65, -0.2);
	 \draw[->] (3.25, 0.1) -- (4.6, 0.9);
	 \draw[<-] (3.25, 0.2) -- (4.6, 1);
   	 \node[] at (4.1, -0.9) {$a_{21}$};
   	 \node[] at (4.57, -0.7) {$a_{12}$};
   	 \node[] at (4.57, -0.05) {$a_{13}$};
   	 \node[] at (4.1, -0.35) {$a_{31}$};
   	 \node[] at (4, 0.9) {$a_{1n}$};
   	 \node[] at (4.5, 0.55) {$a_{n1}$};
\draw (2.5,-1.5) rectangle (5.5, 1.4);
    	\node[above] at (3.4, -1.5) {Mammillary};
	\end{tikzpicture}
\end{center}
\caption{Two bidirected graphs with $n$ compartments (cf.~\cite[Figures~1--2]{singularlocus}).
{ Left:} {\bf Catenary} (path), denoted by ${\rm Cat}_n$.  { Right:} {\bf Mammillary} (star), denoted by ${\rm Mam}_n$. 
}
\label{fig:cat-mam}
\end{figure}
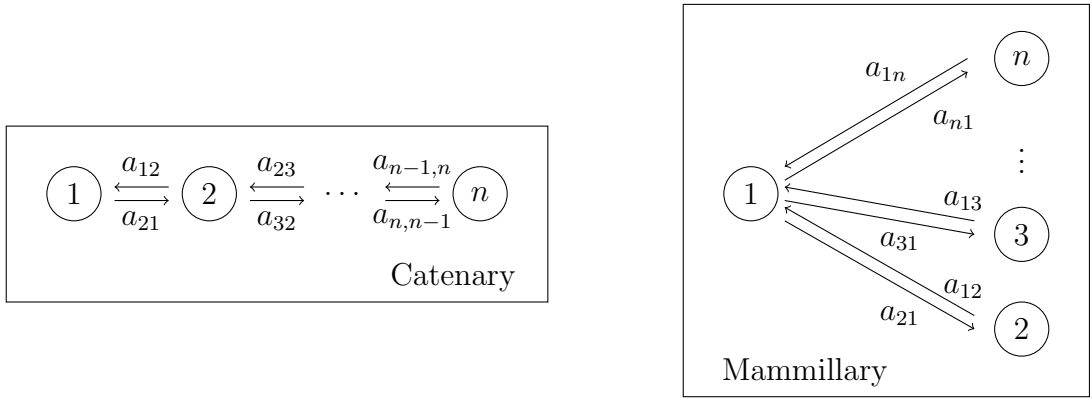

As an easy consequence of Theorem~\ref{thm:class_bidi}, we obtain results on catenary and mammillary models (that is, models in which the underlying graph is, respectively, a path or a star graph, as in Figure~\ref{fig:cat-mam}).  
These results form a substantial improvement over prior results, which largely concerned the case when input and output are equal (see Lemma~\ref{lem:bidi_sameio}).

\begin{cor}[Classification of identifiable catenary models] \label{cor:cat}
Let $n \geq 2$, and let ${\rm Cat}_n$ denote the $n$-compartment catenary graph depicted in 
Figure~\ref{fig:cat-mam}.  
Then a model $({\rm Cat}_n, In, Out, Leak)$ with $|In|=|Out|=1$ 
is generically locally identifiable 
if and only if $|Leak| \leq 1$ and either
(1) $In=Out$ or (2) the input and output compartments are adjacent. 
\end{cor}

\begin{cor}[Classification of identifiable mammillary models] \label{cor:mam}
Let $n \geq 2$, and let ${\rm Mam}_n$ denote the $n$-compartment mammillary graph depicted in 
Figure~\ref{fig:cat-mam}.  Then a model $({\rm Mam}_n, In, Out, Leak)$ with $|In|=|Out|=1$ 
is generically locally identifiable 
if and only if $|Leak| \leq 1$ 
and (at least) one of the following hold:
(1) $In=Out$, (2) $In=\{1\}$, or (3) $Out=\{1\}$.
\end{cor}



%

\subsection{Proof of Theorem~\ref{thm:class_bidi}} \label{sec:proof-tree-model}

To prove Theorem~\ref{thm:class_bidi}, we need two lemmas.  The first pertains to tree models whose identifiability is known from prior results. 

\begin{lemma}\label{lem:bidi_sameio}
If 
$\cm = (G, In, Out, Leak)$ 
is a tree model 
with 
$|Leak| \leq 1$ and 
input and output in a single compartment ($In = Out = \{i\}$),    
then $\cm$ is generically locally identifiable.  
\end{lemma}

\begin{proof}
Let $n$ be the number of compartments. 
Since $In=Out=\{i\}$, $|Leak| \leq 1$, and 
$G$ is inductively strongly connected with respect to $i$, the lemma follows from Proposition~\ref{prop:id-cycle-model-0-1-leaks}.
\end{proof}

The next result, which follows easily from a result in a prior section, 
pertains to when tree models are unidentifiable due to having more parameters than coefficients.

\begin{lemma}[Unidentifiable tree models] \label{lem:unident_bidi}
Let $n \geq 1$. 
Consider a tree model with $n$ compartments, one input, and one output, 
$\cm=(G,  \{in\},  \{out\},  Leak)$. 
If $\dist_G(in,out) \geq 2$ or $|Leak| \geq 2$,
then $\cm$ is unidentifiable.
\end{lemma}

\begin{proof}
As $G$ is a bidirectional tree with $n$ vertices, it has $|E_G|=2n-2$ edges.  
We consider first the case when $|Leak| \geq 2$.  Then $|E_G|+|Leak| \geq (2n-2)+2 = 2n > 2n-1$.  
So, by 
Corollary~\ref{cor:criterion-uniden}, 
$\mathcal{M}$ is unidentifiable.

In the other case, we have $L:=\dist_G(in,out)  \geq 2$.  There are two subcases.  
If $Leak \neq \emptyset$, then 
$|E_G|+ |Leak| \geq (2n-2)+1 > 2n-2 \geq 2n-L$.
If $Leak = \emptyset$, then 
$|E_G|= 2n-2 > 2n-2-1 \geq 2n-L-1$.
In either subcase, by Corollary~\ref{cor:criterion-uniden}, $\mathcal{M}$ is unidentifiable.
\end{proof}

We now prove Theorem \ref{thm:class_bidi}, which we recall states that the implication in Lemma~\ref{lem:unident_bidi} is in fact an equivalence.

\begin{proof}[Proof of Theorem \ref{thm:class_bidi}]
The forward direction $(\Rightarrow)$ is Lemma~\ref{lem:unident_bidi}.

To prove the backward direction $(\Leftarrow)$,  
we first consider the case when $|Leak|=0$.  
If $\dist_G(in,out)=0$, then Lemma \ref{lem:bidi_sameio} implies that $\cm$ is identifiable.  

Now assume that $\dist_G(in,out)=1$ (i.e., $in \leftrightarrows out$ are edges in $G$).  
We will build the bidirectional tree graph $G$ 
by starting with a subtree $G'$ and then successively adding leaf edges.  
The subtree $G'$ comes from removing the edges $in \leftrightarrows out$, which disconnects $G$, and taking
 the component containing $in$.  More precisely, $G'$ is the subgraph induced by 
all $i\in V_G$ such that $\dist_G(in,i) < \dist_G(out,i)$.  
It follows that
$in \in V_{G'}$ and $G'$ is a bidirectional tree. 
So, by Lemma \ref{lem:bidi_sameio}, the model $\cm'=(G',\{in\},\{in\},\emptyset)$ is identifiable. 

Next, let $G''$ be obtained from $G'$ by adding a leaf edge at the input compartment and labeling the new compartment by $out$ (so the new pair of edges is $in \leftrightarrows out$).   By construction, $G''$ is a bidirectional tree and an induced subgraph of $G$.  
Now Proposition~\ref{prop:coeffs-jac_addout_noleak} 
implies that the model 
$\cm''=(G'', \{in\},\{out\}, \emptyset)$ 
is identifiable (because $\cm'$ is).
If $G''=G$, we are done.
If not, we finish building $G$ from $G''$ by adding one leaf edge at a time. 
At each step, 
the graph is a bidirectional tree and an induced subgraph of $G$; and also (by Theorem~\ref{thm:add-leaf-edge}) 
the resulting model with $In=\{in\}$, $Out=\{out\}$, and $Leak=\emptyset$ is identifiable.  
So, as desired, $\cm=(G,\{in\},\{out\},\emptyset)$ is identifiable.

Finally, 
consider the case when $|Leak|=1$.  We already showed that models with  $\dist_G(in,out) \leq 1$ and  $|Leak|=0$ are identifiable, and now Proposition  \ref{prop:add-remove-leak} implies that adding a leak to such models preserves identifiability.  This completes the proof.
\end{proof}


\subsection{Expected dimension of tree models}

Tree models with more than one leak are unidentifiable by Lemma \ref{lem:unident_bidi}, but they have expected dimension for any number of leaks, as long as the input and output are equal or adjacent.

\begin{prop} 
Consider a tree model with exactly one input and one output, $\cm = (G, \{in\}, \{out\}, Leak)$.
If $\dist_G(in,out) \leq 1$, then $\cm$ has expected dimension. 
\end{prop}

\begin{proof} 
Let $n$ be the number of compartments. 
First assume $|Leak| \leq 1$.  By Theorem~\ref{thm:class_bidi}, $\cm$ is generically locally identifiable and so has expected dimension (by Proposition~\ref{prop:MSE}).  In particular, for the model $\overline{\cm}  :=(G, \{in\}, \{out\}, \{i\} )$, the coefficient map, 
which has the form 
 $\bar c : \mathbb{R}^{|E_G|+1}= \mathbb{R}^{2n-1} \to \mathbb{R}^{2n-1}$ 
by Corollary \ref{cor:number-coefficients}, 
has image with dimension equal to $2n-1$.

Now assume $|Leak| \geq 2$.  
By Corollary \ref{cor:number-coefficients}, the coefficient map of $\cm$ has the form 
$c: \mathbb{R}^{|E|+|Leak|} \to \mathbb{R}^{2n-1}$ and (by Theorem~\ref{thm:coeff-i-o-general}) is an extension of $\bar c$ when $i \in Leak$.  Thus, the image of $c$ has dimension equal to $2n-1$, and so $\cm$ has expected dimension.
%
\end{proof}


\subsection{Beyond tree models} \label{sec:what-generalizes-to-strong-con}
Recall that Theorem~\ref{thm:class_bidi} states that a tree model $\cm = (G, \{in\}, \{out\}, Leak) $ is identifiable if and only if 
$\dist_G(in,out) \leq 1$ and $|Leak|\leq 1$.  
It is natural to ask whether any part of this theorem generalizes 
to strongly connected models.  Unfortunately, this is not the case, as the following examples show.

\begin{ex}[Unidentifiable, but $\dist_G(in,out) = 0$ and $|Leak| = 0$]
\label{ex:uniden-dist-0-leak-0-again}
Recall that in the model from Example~\ref{ex:uniden-dist-0-leak-0}, the input and output are equal, and there are no leaks. Nonetheless, the model is unidentifiable.  
\end{ex}

\begin{ex}[Identifiable, but $\dist_G(in,out) = 2$]
\label{ex:iden-dist-2}
In the following model, the distance of the shortest path from input to output is 2, and
\cite[Theorem 3.5]{Gerberding-Obatake-Shiu} implies that the model 
 is generically locally identifiable.
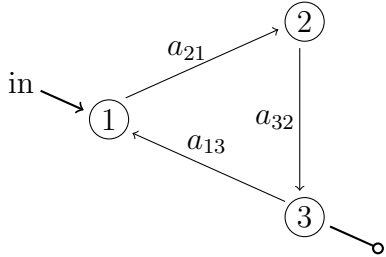
\begin{center}
\begin{tikzpicture}[scale=1.3]
 	\draw (0,0) circle (0.2); 
	\node[] at (0, 0) {1};
 	\draw (2,1) circle (0.2); 
    	\node[] at (2, 1) {2};
 	\draw (2,-1) circle (0.2); 
    	\node[] at (2, -1) {3};
	\draw[->] (.17,.22) -- (1.73,.92) ; 
	\node[] at (.8,.7) {$a_{21}$};
	\draw[<-] (.24,-0.14) -- (1.8,-.84); 
	\node[] at (1,-.254) {$a_{13}$};
	\draw[->] (1.95,.73) -- (1.95,-.73); 
	\node[] at (1.7,0) {$a_{32}$};
	\draw[->,thick] (-.7,.293) -- (-.26,.096);
	\node[] at (-.9,.38) {in};
	\draw[thick] (2.7, -1.293) -- (2.26, -1.096);
	\draw[thick] (2.75, -1.32) circle (.05);
\end{tikzpicture}
\end{center}

\end{ex}

\begin{ex}[Identifiable, but $|Leak| = 2$]
\label{ex:iden-leak-2} 
In the following model, there are 2 leaks and
\cite[Corollary 3.27]{bortner-meshkat} implies that the model 
 is generically locally identifiable.
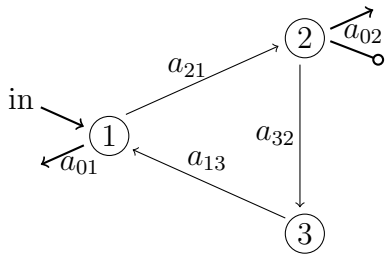
\begin{center}
\begin{tikzpicture}[scale=1.3]
 	\draw (0,0) circle (0.2); 
	\node[] at (0, 0) {1};
 	\draw (2,1) circle (0.2); 
    	\node[] at (2, 1) {2};
 	\draw (2,-1) circle (0.2); 
    	\node[] at (2, -1) {3};
	\draw[->] (.17,.22) -- (1.73,.92) ; 
	\node[] at (.8,.7) {$a_{21}$};
	\draw[<-] (.24,-0.14) -- (1.8,-.84); 
	\node[] at (1,-.254) {$a_{13}$};
	\draw[->] (1.95,.73) -- (1.95,-.73); 
	\node[] at (1.7,0) {$a_{32}$};
	\draw[->,thick] (-.7,.293) -- (-.26,.096);
	\node[] at (-.9,.38) {in};
	\draw[thick] (2.7, 0.8) -- (2.26, 0.97);
	\draw[thick] (2.75, 0.78) circle (.05);
	\draw[<-,thick] (2.7,1.293) -- (2.26,1.096);
	\node[] at (2.6,1.05) {$a_{02}$};
	\draw[<-,thick] (-.7,-.293) -- (-.26,-.096);
	\node[] at (-0.3, -0.3) {$a_{01}$};
\end{tikzpicture}
\end{center}
%
%
%
\end{ex}

In spite of the above examples, we recall 
from Remark~\ref{rmk:relation-bortner-meshkat-uniden-result} 
that strongly connected models (with one input and one output) with $|Leak| \geq 3$ (or, if input equals output, $|Leak| \geq 2$) are unidentifiable.  

\section{Discussion} \label{sec:discussion}

In this work, we made substantial progress on the problem of parameter identifiability for linear compartmental models.  In particular, we expanded the class of linear compartmental models for which structural identifiability can be assessed directly from the underlying graph structure.  While previously this class contained only certain cycle models~\cite{Gerberding-Obatake-Shiu}, 
{ some} inductively strongly connected models, and their generalizations~\cite{bortner-meshkat,MeshkatSullivant, MeshkatSullivantEisenberg}, and was largely focused on the
case where input and output were in the same compartment,  we have now added 
more inductively strongly connected models (Corollary~\ref{cor:addleaf_isc}) and, significantly, 
all tree models with one input and one output with no restrictions on the placement
of the input and output. This includes a complete classification of
identifiability for the much-studied catenary and mammillary models (Theorem \ref{thm:class_bidi}). 

Going forward, a natural problem is to determine what happens 
when there { are multiple leaks} or more than one input or output, or when we go
beyond tree models.  
While Theorem~\ref{thm:class_bidi} does not generalize to all strongly connected models (Section~\ref{sec:what-generalizes-to-strong-con}), a natural first step is to analyze directed-cycle models with one input and one output.  Some partial results are known~\cite{Gerberding-Obatake-Shiu}, but the problem remains open.  Another way to generalize our results is to allow for one-way flow instead of bidirectional flow between compartments in tree models.  One way to accomplish this is to use \cite[Corollary 3.36]{join} to combine bidirectional tree models together over a (one-way) directed edge.  Another possibility is to add leaves to one-way ``path'' models, as in \cite[Proposition 3.29]{bortner-meshkat}.  

Another contribution of our work comes from our results on how to construct new identifiable models from models that are 
previously known to be identifiable (Theorems~\ref{thm:add-leaf-edge} and~\ref{thm:add-leaf}).  We desire more such results and anticipate that they will aid in classifying identifiable models beyond tree models.  A natural first step would be to extend our results on adding leaf edges $i \leftrightarrows n$, where $n$ is a new compartment, to allow new edges of the form $i \to n \to j$, with $i \neq j$, which might be part of a cycle (some related results are ~\cite[Theorem 5.7]{MeshkatSullivant} and ~\cite[Proposition 4.14]{baaijens-draisma}). 
 
Finally, we note that many of our results are proven using our novel combinatorial formula for the coefficients of input-output equations (Theorem \ref{thm:coeff-i-o-general}).  This formula is a new tool for attacking open problems, such as a conjecture concerning the equation of the {\em singular locus} (essentially the locus of unidentifiable parameters) for tree models \cite{singularlocus}.  
 Another potential application of Theorem \ref{thm:coeff-i-o-general} is 
 to the important problem of finding minimal sets of outputs~\cite{Anguelova} (or inputs~\cite{linear-i-o}) for identifiability.

\subsection*{Acknowledgement}
{\footnotesize
This project began at an AIM workshop on ``Identifiability problems in systems biology,'' and the authors thank AIM for providing financial support and an excellent working environment.  
CB and SS 
were partially supported by the NSF (DMS 1615660).
EG was supported by the NSF (DMS 1620109). 
NM was partially supported by the Clare Boothe Luce Program from the Henry Luce Foundation and the 
NSF (DMS 1853525).  
AS was supported by the NSF (DMS 1752672).  
}


\bibliographystyle{plain}
\bibliography{AIM}

\end{document}